\documentclass[12pt]{amsart}
\usepackage[utf8]{inputenc}
\usepackage{amssymb}
\usepackage{amsmath}
\usepackage{amsthm}
\usepackage[T1]{fontenc}
\usepackage{tikz}
\usepackage{tkz-graph}
\usepackage{mathtools}
\usepackage{wasysym}
\usepackage{verbatim}
\usepackage{comment}
\usepackage{graphicx}
\usepackage{color}
\usepackage{xcolor}
\usepackage[margin=1.15in]{geometry}
\usepackage{caption}
\usepackage{subcaption}

\usepackage{geometry,tikz-cd,fancyhdr,wrapfig,enumitem,thm-restate}
\usepackage{hyperref}
\hypersetup{colorlinks=true,citecolor=purple,linkcolor=purple}

\newcommand{\acton}{\rotatebox[origin=c]{90}{ $ \circlearrowleft $ }}
\theoremstyle{plain}
\newtheorem{theorem}{Theorem}[section]
\theoremstyle{definition}
\newtheorem{lemma}[theorem]{Lemma}
\newtheorem{prop}[theorem]{Proposition}

\newtheorem{corollary}[theorem]{Corollary}
\newtheorem{claim}[theorem]{Claim}

\newtheorem{remark}[theorem]{Remark}

\newtheorem{definition}[theorem]{Definition}
\usepackage{style}
\usepackage[style=alphabetic]{biblatex}
\DeclareMathOperator{\de}{def}
\DeclareMathOperator{\perm}{perm}

\newcommand{\mb}[1]{\mathbb{#1}}
\newcommand{\mc}[1]{\mathcal{#1}}

\begin{document}

\title{Dihedral Sieving on Cluster Complexes}
\author[Z. Stier]{Zachary Stier}
\author[J. Wellman]{Julian Wellman}
\author[Z. Xu]{Zixuan Xu}
\date{\today}

\maketitle
\vspace{-0.5cm}
\begin{abstract}
The \emph{cyclic sieving phenomenon} of Reiner, Stanton, and White characterizes the stabilizers of cyclic group actions on finite sets using $q$-analogue polynomials. Eu and Fu demonstrated a cyclic sieving phenomenon on generalized cluster complexes of every type using the $q$-Catalan numbers. In this paper, we exhibit the \emph{dihedral sieving phenomenon}, introduced for odd $n$ by Rao and Suk, on clusters of every type. In the type $A$ case, we show that the Raney numbers count both reflection-symmetric $k$-angulations of an $n$-gon and a particular evaluation of the $q,t$-Fuss--Catalan numbers. We also introduce a sieving phenomenon for the symmetric group, and discuss possibilities for dihedral sieving for even $n$.
\end{abstract}

\section{Introduction}

The {\em cyclic sieving phenomenon} was introduced by V. Reiner, D. Stanton, and D. White in \cite{RSW2004cyclic} as a polynomial method for enumerating fixed points of combinatorial objects. The idea is generally, for each object in a family of combinatorial objects admitting some cyclic group action and for each element of that cyclic group, to assign a polynomial such that its evaluation at a well-chosen complex number exactly equals the number of fixed points under that group element's action. Formally speaking, suppose $X$ is a finite set acted on by a cyclic group $C_n = \langle r \rangle$, and $X(q)$ is a polynomial in $q$. The pair $(X\ \acton \ C_n,X(q))$ exhibits the \emph{cyclic sieving phenomenon} (CSP) if, for all $\ell \in [n]$, 
\[|\{x\in X : r^\ell x = x \}| = X(e^{2\ell \pi i/n}) .\]

The cyclic sieving phenomenon was generalized to finite sets acted on by a product of two cyclic groups, or \emph{bicyclic sieving}, by H. Barcelo, V. Reiner, and D. Stanton in \cite{BHR2008bicyclic}, and was later generalized to any finite abelian group written explicitly as a product of cyclic groups in \cite{bergetEuReiner2011}. 

Furthermore, S. Rao and J. Suk in \cite{rs2017dihedral} extended the sieving phenomenon to the dihedral group $I_2(n)$ for $n$ odd, where $I_2(n) \coloneqq \<r,s|r^n = s^2 = 1,\ rs = sr^{-1} \>$. 
Let the \emph{defining representation} $\rho_{\de}$ be the two-dimensional representation of $I_2(n)$ sending
\begin{align*}
    r &\mapsto \begin{pmatrix} \cos\frac{2\pi}{n} & \sin\frac{2\pi}{n} \\ -\sin\frac{2\pi}{n} & \cos\frac{2\pi}{n} \end{pmatrix} \\
    s &\mapsto \begin{pmatrix} 1 & 0 \\ 0 & -1 \end{pmatrix}.
\end{align*}
Then odd dihedral sieving is defined as follows. Suppose $X$ is a finite set acted on by the dihedral group $I_2(n)$ with $n$ odd, and $X(q,t)$ is a symmetric polynomial in $q$ and $t$. The pair $(X \ \acton \ I_2(n),X(q,t))$ has the \emph{dihedral sieving phenomenon} (DSP) if, for all $g \in I_2(n)$ with eigenvalues $\{\lambda_1,\lambda_2\}$ for $\rho_{\de}(g)$, 
\[|\{x \in X : gx = x\}| = X(\lambda_1,\lambda_2).\]

Note that we require $X$ to be a symmetric polynomial since the eigenvalues $\lambda_1,\lambda_2$ of $\rho_{\de}$ are not ordered. Rao and Suk showed instances of odd dihedral sieving using the \emph{$q,t$-analogues} $\{n\}_{q,t},\{n\}_{q,t}!,$ and $\binomqt{n}{k}_{q,t}$ defined as follows (under slightly different definitions):
\begin{align*}
\{n\}_{q,t} &\coloneqq q^{n-1}+q^{n-2}t+\cdots+t^{n-1} = \sum\limits_{i=0}^{n-1}q^it^{n-1-i} \\
\{n\}!_{q,t} &\coloneqq \{n\}_{q,t}\cdot\{n-1\}_{q,t}\cdot\cdots\cdot\{1\}_{q,t} = \prod\limits_{i=1}^n\{i\}_{q,t} \\
\binomqt{n}{k}_{q,t} &\coloneqq \frac{\{n\}!_{q,t}}{\{k\}!_{q,t}\{n-k\}!_{q,t}}
\end{align*}

The above polynomials correspond to the \emph{Fibonacci polynomials}~\cite{HogLon74} and \emph{Fibonomial coefficients}~\cite{ACMS14} in $X$ and $Y$ used in~\cite[Proposition 3.5]{rs2017dihedral}, through the substitutions $X\mapsto q+t$ and $Y\mapsto -qt$. The following theorem summarizing their instances of dihedral sieving also uses the $q,t$-Catalan numbers $\Cat_n(q,t)$ as defined by A. Garsia and M. Haiman in~\cite{garsiahaimann1996remarkable}. 

\begin{theorem}[\cite{rs2017dihedral}]\label{thm:prevdisieve}
Let $n$ be odd and let $0\le k\le n$. Then the following exhibit dihedral sieving: 
\begin{enumerate}
    \item $\lpr{X \ \acton \ I_2(n), \binomqt{n}{k}_{q,t}}$ for $X$ the set of order-$k$ subsets of the positive integers $[n]$. 
    
    \item $\lpr{X \ \acton \ I_2(n), \binomqt{n+k-1}{k}_{q,t}}$ for $X$ the set of order-$k$ multisubsets of the positive integers $[n]$. 
    
    \item $\lpr{X \ \acton \ I_2(n), \frac{1}{\{n+1\}_{q,t}}\binomqt{2n}{n}_{q,t}}$ for $X$ the set of non-crossing partitions of the $n$-gon. 
    
    \item $\lpr{X \ \acton \ I_2(n), \frac{1}{\{n\}_{q,t}}\binomqt{n}{k}_{q,t}\binomqt{n}{k+1}_{q,t}}$ for $X$ the set of non-crossing partitions of the $n$-gon using $n-k$ blocks. 
    
    \item $\lpr{X \ \acton \ I_2(n), (qt)^{\binom{n-2}{2}}\Cat_{n-2}(q,t)}$ for $X$ the set of triangulations of the $n$-gon. 
\end{enumerate}
\end{theorem}
Another example of dihedral sieving is the action generated by \emph{promotion} and \emph{complementation} of plane partitions using MacMahon product formulas. This result was first shown in \cite{abuzzahabcyclic} and \cite{rhoades2010cyclic}, and is summarized in \cite[Theorem 1.4]{hopkins2019cyclic}.

\subsection{Statement of Results}

As noted in \thmref{thm:prevdisieve}, dihedral sieving for triangulations of an $n$-gon has already been demonstrated in \cite[Theorem 6.1]{rs2017dihedral}, using the $q,t$-Catalan numbers. Our results generalize this in two different directions. First, we consider pentagonalizations, heptagonalizations, and so on (in general, odd-$k$-angulations). 
For this, we will use the $q,t$-Fuss--Catalan numbers (see \defref{def:qtratcat}).

\begin{theorem}
\label{thm:Akangulations} 
Let $X$ be the set of $k$-angulations of an $n$-gon, for $n$ odd and $n \equiv 2\ (\mod k-2)$. Then the pair $(X \ \acton \ I_2(n),\Cat_{n-1,\frac{n-2}{k-2}}(q,t))$ exhibits dihedral sieving.
\end{theorem}

Next, we will consider triangulations corresponding to other types. For this we will use the cluster complex for each type as an algebraic interpretation of triangulations. The result below is phrased in this language. Let $X = \Delta(\Phi)$ be a cluster complex, with dihedral action generated by the reflections $\tau_+$ and $\tau_-$, operators defined by S. Fomin and A. Zelevinsky in \cite{fz2003Ysystems}. 
We will use a generalization of the $q,t$-Catalan numbers to arbitrary type (see \defref{def:qtcatW}), which agrees with the usual $q,t$-Catalan numbers in type $A$. 

The following theorem relies on two conjectures of C. Stump~\cite{stump2010qtfcat} about the values of the $q,t$-Catalan numbers for arbitrary type in the specializations $q = t^{-1}$ and $t = 1$. The theorem will hold for any polynomial which has these specializations.

\begin{theorem}\label{thm:alltypesoftriangles}
Assuming two conjectures of C. Stump, the pair $(\Delta(\Phi) \ \acton \ I_2(n),\Cat(\Phi,q,t))$ exhibits dihedral sieving for all odd $n$ and $\Phi$ of type $A,B/C,D,E,F$, or $I$.
\end{theorem}
We will explain later in \secref{sec:othertypes} that the remaining types $H_3$ and $H_4$ do not admit an odd dihedral group action to obtain sieving.

From a purely enumerative combinatorics point of view, our result stated in \thmref{thm:Akangulations} is a direct generalization of \thmref{thm:prevdisieve}(5) given in \cite{rs2017dihedral} by showing the odd dihedral sieving phenomenon on $k$-angulations and the $q,t$-Fuss--Catalan numbers, since $k$-angulations are direct generalizations of triangulations and the $q,t$-Fuss--Catalan numbers directly generalize the $q,t$-Catalan numbers.

From a broader perspective, it was previously known that $q,t$-Catalan numbers have deep connections with a number of combinatorial settings, such as invariant theory and Hilbert schemes of unordered points in $\C^2$ in \cite{Haiman}, knot theory in \cite{GorskyMazin,GalashinLam}, and rational Cherednik algebras in \cite{GordonGriffeth}. However very little was known about the connections between $q,t$-Catalan numbers and cluster algebras. Our work generalizing the results in \thmref{thm:prevdisieve}(5) further hints towards the existence of a connection between cluster theory and $q,t$-Catalan numbers. 

The paper is organized as follows. First, we will discuss various interpretations of sieving phenomena in \secref{sect:phenom}. In \secref{sec:ccdefs} we will define the cluster complex for any root system, and discuss the realization of the cluster complex for types $B$ and $D$. Then we will proceed to proving \thmref{thm:Akangulations} in \secref{sec:kangA}, and \thmref{thm:alltypesoftriangles} in \secref{sec:othertypes}. Finally, in \secref{sec:evenmore}, we discuss some miscellaneous conjectures and describe possible notions of dihedral sieving for even $n$.

\section{Preliminaries}
In what follows, we lay out definitions and past results which will be relevant in the sections that follow. 

\subsection{Sieving Phenomena}\label{sect:phenom}

There is an equivalent definition of cyclic sieving using representation theory, which we now describe in order to give us access to past results framed in that manner. In order to do so, first we will define the ring of representations.

\begin{definition}[Representation Ring]\label{def:repring}
Let $G$ be a finite group, and let $\irr G$ be the set of isomorphism classes of irreducible finite-dimensional $\mb C$-representations of $G$. The \emph{representation ring} $\rep G$ is the
free abelian group $\mb Z[\irr G]$ having basis elements $[V]$ indexed by $V$ in $\irr G$, and multiplication defined by $[V]\cdot [W]=\sum_{i}
[V_i]$ if  $V\otimes W \cong \bigoplus_i V_i$
with $V_i \in \irr G$.
\end{definition} 

In other words, elements of $\rep G$ are {\it virtual} finite-dimensional representations of $G$, with addition via direct sum, and multiplication induced from the tensor product. 

Let $\rho$ be a one-dimensional representation of $C_n$, which can be specified by choosing a generator of $C_n$ to send to $e^{2\pi i /n}$, as is done in the definition of cyclic sieving. 
\begin{prop}[{\cite[Proposition 2.1]{RSW2004cyclic}}]
A pair $(X \ \acton \ C_n,X(q))$ has the CSP if and only if $X(\rho)$ and $\mb C[X]$ are isomorphic as $C_n$-representations.
\end{prop}
Here we let $X(\rho)$ be the evaluation of $X(\cdot)$ at $\rho$, direct sums taking the role of addition and tensor powers taking the role of exponentiation, and $\C[X]$ being the $\C$-vector space formally generated by $X$'s elements, with natural action $\C[X] \ \acton \ C_n$ given canonically by $X \ \acton \ C_n$.

This gives an equivalent description of the cyclic sieving phenomenon in the language of representations of $C_n$. Rao and Suk, motivated by the above proposition, defined sieving phenomena for any group along with a
choice of a finite generating set for its representation ring.
\begin{definition}[{\cite[Definition 2.7]{rs2017dihedral}}]\label{def:rsGsieve}
Let $G$ be a group acting on a finite set $X$, and let $\{\rho_1,\ldots,\rho_k\}$ be a generating set for the representation ring $\rep G$. Together with a $k$-variable polynomial $X_{RS}(q_1,\ldots,q_k)$, these form a triple $(X\ \acton \ G,\{\rho_1,\ldots,\rho_k\},X_{RS}(q_1,\ldots,q_k))$ which exhibits \emph{$G$-sieving} if and only if $X_{RS}(\rho_1,\ldots,\rho_k)$ and $\mb C[X]$ are isomorphic as $G$-representations.
\end{definition}

This definition encompasses all known forms of sieving phenomena
after applying \lemref{lem:equivdef}. 
For example, our initial notion of dihedral sieving can be seen as the special case $G = I_2(n)$, with $\rep I_2(n)$ generated by $\rho_{\de}$ and $-\det$, where $\det$ is the determinant of $\rho_{\de}$. We use the subscript ``RS'' to distinguish the polynomials Rao and Suk use from the ones in our alternative definition, which can also describe all previous forms of sieving. 
\begin{definition}[$G$-sieving]\label{def:genSieve}
Suppose a group $G$ acts on finite set $X$, and $X(q_1,\ldots,q_d)$ is a symmetric polynomial in $d$ variables with $\rho$ a $d$-dimensional representation of $G$. The triplet $(X\ \acton \ G,\rho,X(q_1,\ldots,q_d))$ exhibits \emph{$G$-sieving} if and only if for all $g\in G$, if $\lambda_1,\ldots,\lambda_d$ are the eigenvalues of $\rho(g)$, then \[|\{x \in X: gx = x \}| = X(\lambda_1,\ldots,\lambda_d). \]
\end{definition}

Our initial notion of dihedral sieving 
is a special case of our new definition, namely, with $\rho = \rho_{\de}$ the defining representation and $G$ a dihedral group $I_2(n)$ with $n$ odd. 

\begin{lemma}[cf. {\cite[Proposition 4.3]{rs2017dihedral}}]\label{lem:equivdef}
Suppose that $n$ is odd, and that $I_2(n)$ acts on a finite set $X$. Then for any two-variable polynomial $X_{RS}(\cdot,\cdot)$, the $I_2(n)$-representations $\mb C[X]$ and $X_{RS}(\rho_{\de},-\det)$ are isomorphic --- that is, $(X\ \acton\ I_2(n),\{\rho_{\de},-\det\},X_{RS}(q_1,q_2))$ exhibits the DSP in the sense of \defref{def:rsGsieve} --- if and only if $(X\ \acton \ I_2(n),\rho_{\de},X_{RS}(q+t,-qt))$ exhibits the DSP in the sense of \defref{def:genSieve}.
\end{lemma}\begin{proof}
Fix $n$ odd, and a finite set $X$ with an $I_2(n)$ action. Then for all $g \in I_2(n)$, we have $|\{x\in X : gx = x \}| = \tr (\mb C[X](g))$, since the trace counts the number of fixed points of the permutation represention. Next, for a polynomial $X_{RS}$ in two variables, we have $\tr X_{RS}(\rho_{\de},-\det)(g) = X_{RS}(\tr \rho_{\de}(g),-\tr \det(g)) = X_{RS}(\lambda_1 + \lambda_2,-\lambda_1\lambda_2)$, where $\lambda_1,\lambda_2$ are the eigenvalues of $\rho_{\de}(g)$ in some order. Finally, since $\mb C[X] \cong X_{RS}(\rho_{\de},-\det)$ if and only if they have the same trace, the claim follows. 
\end{proof}

With \lemref{lem:equivdef} in hand, we can show instances of dihedral sieving for odd $n$ using our initial notion of dihedral sieving, without needing to consider the representation-theoretic interpretation.

\subsection{Generalizations of Catalan numbers}

\subsubsection{Raney numbers}
The following discussion sets up the notion of {\em Raney numbers}, which provide a convenient intermediate step in establishing our main counting results. 
\begin{definition}[Raney Numbers]
The {\em Raney numbers} are defined as follows: 
\[R_{p,r}(k) \coloneqq \frac{r}{kp+r}\binom{kp+r}{k}.\]
\end{definition}

\begin{remark}
When $r = 1$, the Raney number is exactly the Fuss--Catalan number
\[R_{p,1}(k) = \frac{1}{(p-1)k+1}\binom{kp}{k}.\]
If $r = 1, p = 2$, the Raney number is the usual Catalan number:
\[R_{2,1}(k) = \frac{1}{k+1}\binom{2k}{k}\]
which is the number of triangulations of a ($k+2$)-gon. This is a special case of the Fuss--Catalan numbers counting the number of $(p+1)$-angulations of a ($k+2$)-gon. 
Sometimes the alternate notations $C_p(k)$ and $C_k$ are used for the Fuss--Catalan and Catalan numbers, respectively, but for clarity we shall only use $C_i$ in this paper to denote a cyclic group. 
\end{remark} 

Raney numbers count a combinatorial object called \emph{coral diagrams}. 

\begin{definition}[Coral Diagram]
    A {\em coral diagram of type $(p,r,k)$} is a rooted tree obtained by taking an $r$-star centered at the root (that is, assigning the root $r$ immediate children) and then, $k$ times, choosing a leaf to assign $p$ new children (and the chosen leaf may be one which was added previously in the construction).
\end{definition}
See \figref{fig:coral ex} for an example of a coral diagram of type $(2,4,4)$. 
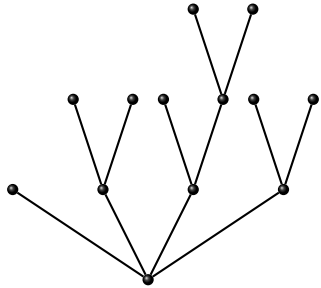
\begin{figure}
    \centering
\begin{tikzpicture}[scale=1.2]
	\SetFancyGraph
	\Vertex[NoLabel,x=0,y=0]{0}
	\Vertex[NoLabel,x=-1.5,y=1]{1}
	\Vertex[NoLabel,x=-0.5,y=1]{2}
	\Vertex[NoLabel,x=0.5,y=1]{3}
	\Vertex[NoLabel,x=1.5,y=1]{4}
	\Vertex[NoLabel,x=-0.5-0.33,y=2]{5}
	\Vertex[NoLabel,x=-0.5+0.33,y=2]{6}
	\Vertex[NoLabel,x=0.5-0.33,y=2]{7}
	\Vertex[NoLabel,x=0.5+0.33,y=2]{8}
	\Vertex[NoLabel,x=1.5-0.33,y=2]{9}
	\Vertex[NoLabel,x=1.5+0.33,y=2]{10}
	\Vertex[NoLabel,x=0.5,y=3]{11}
	\Vertex[NoLabel,x=0.5+2*0.33,y=3]{12}
	
	\Edges[style={thick}](0,1)
	\Edges[style={thick}](0,2)
	\Edges[style={thick}](0,3)
	\Edges[style={thick}](0,4)
	\Edges[style={thick}](2,5)
	\Edges[style={thick}](2,6)
	\Edges[style={thick}](3,7)
	\Edges[style={thick}](3,8)
	\Edges[style={thick}](4,9)
	\Edges[style={thick}](4,10)
	\Edges[style={thick}](8,11)
	\Edges[style={thick}](8,12)
\end{tikzpicture}
    \caption{An example of a coral diagram of type $(2,4,4)$.}
    \label{fig:coral ex}
\end{figure}

\begin{prop}[{\cite[Theorem 2.5]{beagley2015raney}}]\label{prop:coralRaney}
The Raney number $R_{p,r}(k)$ is equal to the number of coral diagrams of type $(p,r,k)$.
\end{prop}

The following two recursions on Raney numbers will be called upon in the sequel. 

\begin{prop}[{\cite[Theorem 2.6]{hilton1991catalan}}] \label{prop:RaneyRecursion}
Let $p$ be a positive integer and let $r,k$ be nonnegative integers. Then we have
\[R_{p,r}(k) = \sum_{i_1+\cdots+i_r = k}R_{p,1}(i_1)R_{p,1}(i_2)\cdots R_{p,1}(i_r).\]
\end{prop}

\begin{prop}[{\cite[Theorem 2.2]{raneynumber}}]\label{prop:RaneyRecursion2}
The Raney numbers satisfy the following recurrences: 
\begin{align*}
R_{p,1}(k) &= \sum\limits_{i=0}^{k-1}R_{p,1}(i)R_{p,p-1}(k-1-i) \\
R_{p,r}(k) &= \sum\limits_{i=0}^kR_{p,r}(i)R_{p,r-1}(k-i) & \text{for $r>1$.}
\end{align*}
\end{prop}

\subsubsection{$q,t$-Fuss--Catalan numbers}
The following discussion provides a framework which we use to construct the polynomials in $q$ and $t$ necessary for our proof of the odd dihedral sieving phenomenon. As the name suggests, this is a further generalization from the Fuss--Catalan and $q$-Fuss--Catalan numbers, albeit in a less straightforward manner. The definitions are due in large part to \cite{garsiahaimann1996remarkable}. 

\begin{definition}[Dyck Path]
In a grid $a$ boxes high and $b$ boxes wide, an \emph{$(a,b)$-Dyck path} is a path, entirely going either north (N) or east (E), along the edges in the graph that do not cross below the long diagonal from the bottom-left (southwesternmost) corner to the upper-right (northeasternmost) corner. 
\end{definition}

A Dyck path gives rise to a Young diagram by considering the squares in the grid completely above the path; see, for instance, \figref{fig:tab ex} for a Dyck path drawn in the Young diagram. 
\begin{definition}[Area]
The \emph{area} of an $(a,b)$-Dyck path is the number of full grid squares between the path and the diagonal. Equivalently, it is the number of squares in the Young diagram below the path. We denote it as $\area(\gl)$. 
\end{definition}

In \figref{fig:tab ex}, the path shown has area 11. 

\begin{figure}
     \centering
     \begin{subfigure}[b]{0.53\textwidth}
         \centering
         \includegraphics[width=\textwidth]{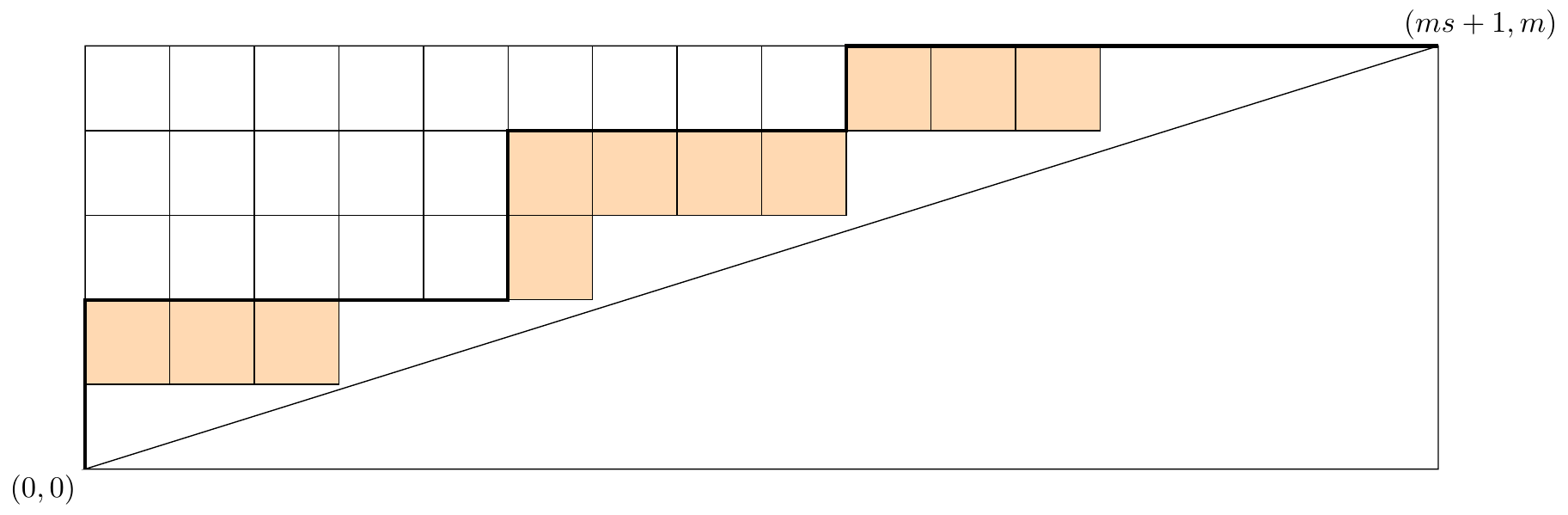}
         \caption{}
         \label{fig:tab ex}
     \end{subfigure}
     \hfill
     \begin{subfigure}[b]{0.44\textwidth}
         \centering
         \includegraphics[width=\textwidth]{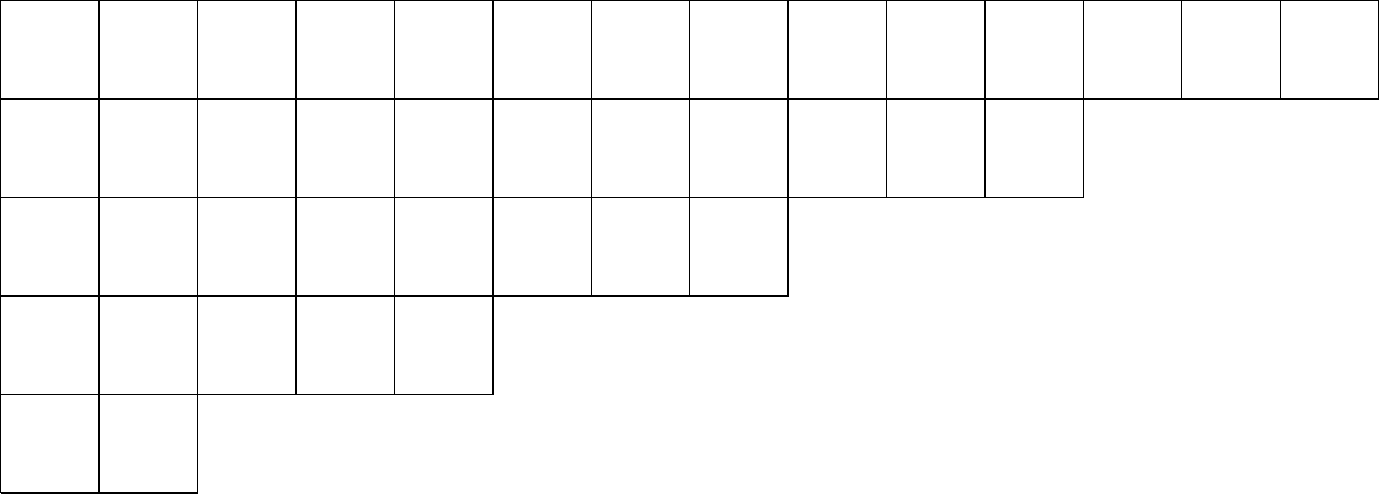}
         \caption{}
         \label{fig:D325}
     \end{subfigure}
        \caption{(a) Young diagram $Y_3(0,5)$ with a Dyck path and corresponding area shown. Note that since $\gcd(ms+1,m)=1$, the diagonal lies strictly below the diagram. (b) Young diagram $Y_3(2,5)$.}
        \label{fig:two young diagrams}
\end{figure}

\begin{definition}[Sweep, cf. \cite{ALW2016ratpf}]
The \emph{sweep} function is a function on $(a,b)$-Dyck paths, defined recursively as follows. Assign the level $\ell_0 = a$ to the first edge in the path $\gl$. Each subsequent edge has level $\ell_i = \ell_{i-1} + a$ if the $i$th step is North and $\ell_i = \ell_{i-1} - b$ if the $i$th step is East. The edges are then sorted by increasing level to obtain a new path in the grid. We denote it by $\sweep(\gl)$. 
\end{definition}

By \cite{ALW2015Sweep, TW2015Sweep}, $\sweep$ is a bijection on $(a,b)$-Dyck paths. 

\begin{definition}[Rational $q,t$-Catalan Numbers]\label{def:qtratcat}
The \emph{rational $q,t$-Catalan numbers} $\Cat_{a,b}(q,t)$ are defined by the following sum over $(a,b)$-Dyck paths $\gl$: 
$$\Cat_{a,b}(q,t) \coloneqq \sum\limits_{\gl}q^{\area(\gl)}t^{\area(\sweep(\gl))}.$$
\end{definition}

Note that when $t=1$, $\sweep(\gl)$ is irrelevant in the evaluation of $\Cat_{a,b}(q,1)$. 

We assume henceforth that $s$ and $m$ are odd positive integers. 

\begin{definition}
Let $Y_s(\ell,m)$ be the Young diagram for the partition $(\ell+\gl_1,\dots,\ell+\gl_m)$ for $\gl_k \coloneqq (m-k)s$. 
\end{definition}

See \figref{fig:D325} for $Y_3(2,5)$. 

This modification of ``regular'' Young diagrams will prove to be a convenient intermediate, inductive object to study later. 

\begin{lemma}
The zero-area Dyck path in the $(ms+1)\times m$ grid corresponds to the Young diagram $Y_s(0,m)$: zero squares on the bottom-most row, $s$ squares on the next, etc., and $(m-1)s$ squares on the top row. 
\end{lemma}
\begin{proof}
The number of squares on row $i$ (indexing from 0 to $m-1$) is equal to $\floor{n_i}$ where $n_i$ satisfies $\frac{i}{n_i}=\frac{m}{ms+1}$, the slope of the long diagonal, thus $n_i=\lpr{\frac{ms+1}{m}}i=si+\frac{i}{m}$. Since $0\le i<m$, $\floor{n_i}=si$, as desired. \end{proof}

See \figref{fig:tab ex} for an example of the diagram for $m=5$ and $s=7$.

\begin{definition}
$Y_s(\ell,m)$ has the indexing $(i,j,k)$ of its boxes as follows. $1\le i\le m$ tells the component of the partition, equivalently the row in the Young diagram. $0\le j\le m-i$ tells the ``block'' in the row, from the left: either the block of width $\ell$, or one of the blocks of width $s$. $1\le k\le\begin{cases}
\ell & j=0 \\
s & j>0
\end{cases}$ tells the position in the block, from the left. 
\end{definition}
For instance, in \figref{fig:tab ex}, highlighted squares include: $(1,4,2)$, $(2,2,3)$, $(2,3,2)$, $(3,2,3)$, $(4,1,1)$. In \figref{fig:D325} the southwesternmost square is indexed as $(5,0,1)$.

\subsection{Cluster Complexes}\label{sec:ccdefs}

In this section, we will be reviewing some key definitions and facts needed to define sieving for cluster complexes. Then, in \secref{sec:othertypes}, we will be prepared to show the odd dihedral sieving phenomenon for each cluster complex. 

\begin{definition}[Reduced Root System]
A subset $\Phi$ in a Euclidean space $E$ endowed with inner product $(\cdot,\cdot)$ is called a \textit{reduced root system} if the following axioms are satisfied:
\begin{enumerate}
    \item $\Phi$ is finite, spans $E$ and $0\notin \Phi$.
    \item If $\alpha\in \Phi$, then the only multiples of $\alpha$ in $\Phi$ are $\pm \alpha$.
    \item If $\alpha\in \Phi$, the reflection across $\alpha$ defined as $\sigma_\alpha(\beta) = \beta - \frac{2(\beta,\alpha)}{(\alpha, \alpha)}$ leaves $\Phi$ invariant.
\end{enumerate}
\end{definition}

\begin{definition}[Simple Roots]\label{def:rootsystem}
A subset $\Pi$ of a root system $\Phi$ is a set of \textit{simple roots} if:
\begin{enumerate}
    \item $\Pi$ is a basis of $E$.
    \item each root $\beta\in \Phi$ can be written as \begin{equation}\beta = \sum k_\alpha \alpha\tag{$*$}\label{eq:root system expansion}\end{equation} where $\alpha\in \Pi$ and $k_\alpha$ are all nonpositive or all nonnegative. 
\end{enumerate}
\end{definition}

A root system is called \emph{irreducible} if it cannot be partitioned into the union of two proper subsets such that each root in one set is orthogonal to each root in the other. From now on, let $\Phi$ be an irreducible root system of rank $n$. Let $\Phi_{>0}$ denote the set of {\em positive roots} in $\Phi$ with respect to the simple roots $\Pi = \{\alpha_1,\ldots, \alpha_n\}$, i.e. those roots with nonnegative coefficients in the expansion \eqref{eq:root system expansion}. Correspondingly, let $-\Pi= \{-\alpha_1,\ldots, -\alpha_n\}$ denote the set of {\em negative simple roots}. Let $S = \{s_{\alpha_1},\ldots, s_{\alpha_n}\}$ denote the set of simple reflections in $\Aut(E)$ corresponding to reflecting across the simple roots $\alpha_i$. Note that $S$ generates a finite reflection group $W\le\Aut(E)$ that acts on $\Phi$ naturally. Lastly, let $\Phi_{\ge -1} = -\Pi\cup \Phi_{>0}$ denote the set of almost positive roots. Now we proceed to define an action on $\Phi_{\ge -1}$. 

Let $I_+\sqcup I_-$ be a partition of $\Pi$ such that the roots in each of $I_+$ and $I_-$ are mutually orthogonal. Define
\[\tau_{\epsilon}(\alpha) = \begin{cases}
\alpha & \text{if }\alpha = -\alpha_i\text{ for } i\in I_{-\epsilon}\\
\lpr{\prod_{i \in I_\epsilon}s_i}(\alpha) & \text{otherwise}
\end{cases}
\]
for $\epsilon\in \{+,-\}$. 
The product $R = \tau_-\tau_+$ generates a cyclic group $\la R \ra$ that acts on $\Phi_{\ge -1}$. 

Then we can use the map $R$ to define a relation of \textit{compatibility} on $\Phi_{\ge -1}$ by
\begin{enumerate}
    \item $\alpha,\beta\in \Phi_{\ge -1}$ are compatible if and only if $ R(\alpha),R(\beta)$ are compatible;
    \item $-\alpha_i\in -\Pi$ and $\beta\in \Phi_{>0}$ are compatible if and only if the simple root expansion of $\beta$ does not include $\alpha_i$ with a nonzero coefficient.
\end{enumerate}

For any two distinct simple roots $\alpha,\beta\in \Phi_{\ge -1}$, we can repeatedly let $R$ act on them until one is a negative simple root, and check whether the two resulting roots satisfy condition (2) above. Thus the two conditions are sufficient to determine compatibility for any two almost positive roots. Condition (2) is invariant under powers of $R$, so the definition is also consistent; this is stated by Fomin and Reading as \cite[Theorem 2.2]{fr2005generalized}, as a summary of \cite[Section 3.1]{fz2003Ysystems}. 

Now we are ready to define the cluster complex.

\begin{definition}[Cluster Complex]
The {\em cluster complex} $\Delta(\Phi)$ of the root system $\Phi$ as above is the simplicial complex whose faces are the sets of pairwise compatible roots in $\Phi_{\ge -1}$.
\end{definition}
For $\Phi$ in type $A_n, B_n$, or $D_n$, the cluster complex $\Delta(\Phi)$ can be realized as dissections of a regular polygon such that $\la R \ra$ corresponds to a group action on the dissections by rotating the given polygon. We will present the bijections between cluster complexes of these types and dissections of regular polygons in \secref{sec:othertypes}.

\section{Dihedral Sieving on $k$-Angulations of $n$-gons
}\label{sec:kangA}

In this section, we will show the odd dihedral sieving phenomenon on what turn out to be (generalized) clusters of type $A$. In particular, we will prove \thmref{thm:Akangulations}, restated as the following: 

\setcounter{section}{1}
\setcounter{theorem}{1}
\begin{theorem}
Let $X_{s,m}$ with $s$ and $m$ odd be the set of $(s+2)$-angulations of an $n$-gon for $n \coloneqq sm+2$, with the natural dihedral action of $I_2(n)$. Then $(X_{s,m}\ \acton \ I_2(n),\Cat_{sm+1,m}(q,t))$ exhibits dihedral sieving.
\end{theorem}
\setcounter{section}{3}
\setcounter{theorem}{0}

We begin by proving general results about counting angulations fixed under certain dihedral group actions, nicely represented using Raney numbers. Then, we show that the relevant specializations of the $q,t$-Catalan numbers coincide with those same Raney number expressions. This completes the task of counting the appropriate fixed clusters using the $q,t$-Catalan numbers. 

\begin{definition}[$k$-angulation]
A \emph{$k$-angulation} of an $n$-gon is a is a subset of the set of interior diagonals of the $n$-gon that dissects the $n$-gon into $k$-gons.
\end{definition}

\begin{figure}
    \centering
    \includegraphics[scale=1.5]{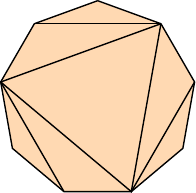} \ \includegraphics[scale=1.5]{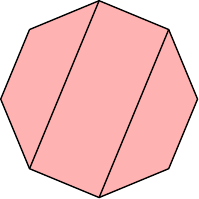} \
    \includegraphics[scale=1.5]{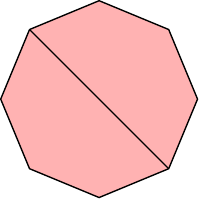} \ \includegraphics[scale=1.5]{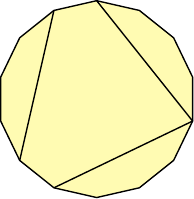}
    \caption{Four examples of $k$-angulations with rotational or reflective symmetry.}
    \label{fig:k-ang example}
\end{figure}

\begin{lemma}
$n$-gons that admit a $k$-angulation satisfy $n\equiv2\pmod{k-2}$. 
\end{lemma}
\begin{proof}
Induction on the number of parts of the $k$-angulation. 
\end{proof}

We will often write $n=m(k-2)+2$, where we may interpret $m$ as the number of $k$-gons resulting from any $k$-angulation of the $n$-gon. We will also frequently use $s \coloneqq k-2$ to keep expressions cleaner. 

\begin{corollary}
For odd $n$ and even $k$, no $n$-gon admits a $k$-angulation. 
\end{corollary}

The dihedral group element $g\in I_2(n)$ acts on a $k$-angulation by sending vertex $i$ to $g(i)$ and correspondingly sends the chord connecting $i$ and $j$ to the chord connecting $g(i)$ and $g(j)$. It is evident that this action sends $k$-angulations to $k$-angulations. In the case of $n$ odd, there is a unique ``type'' of reflection in $I_2(n)$, which has precisely one fixed point and may be viewed geometrically as reflection about the perpendicular bisector of some side of the $n$-gon. This stands in contrast to the case of $n$ even, where the two ``types'' of reflection have zero or two fixed points, as seen in \figref{fig:dihedral ex}. 

\begin{figure}
    \centering
    \includegraphics[scale=1.5]{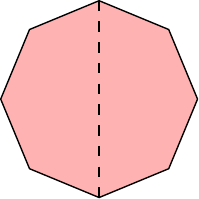} \includegraphics[scale=1.5]{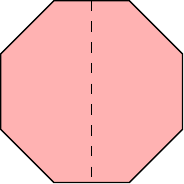} \includegraphics[scale=1.5]{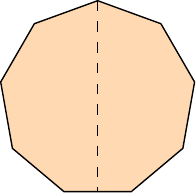}
    \caption{Instances of dihedral group reflections: two (in red) for $n$ even and one (in orange) for $n$ odd.}
    \label{fig:dihedral ex}
\end{figure}

See e.g.~\cite[\S2.1]{eufu2008dissections} for a discussion of the correspondence between $k$-angulations and type A cluster complexes, which respects the dihedral action. From this, we shall now focus on studying $k$-angulations, which immediately gives the desired results about the relevant cluster complexes. 

In the following two results, we explicitly enumerate the fixed points of the dihedral group's action, in terms of Raney numbers. Subsequently we prove that the Raney numbers thus obtained also equal the relevant specializations of the $q,t$-Catalan numbers. 

\begin{theorem}\label{thm:reflectioncounting}
Let $F_k(n)$ be the number of $k$-angulations of an $n$-gon fixed by the reflection across a specific axis for $n$ and $k$ odd. Then
\[F_k(n) = R_{k-1, \frac{k-1}{2}}\left(\frac{m-1}{2}\right)\]
where $m = \frac{n-2}{k-2}$.
\end{theorem}

\begin{proof}
By~\propref{prop:coralRaney} it suffices to biject with coral diagrams of type $(k-1,\frac{k-1}{2},\frac{m-1}{2})$. In any $k$-angulation, some $k$-gon contains the center of the $n$-gon. This $k$-gon is symmetric with respect to the axis of reflection, and both $n$ and $k$ are odd, so the central $k$-gon includes both the edge and vertex intersected by the axis of reflection. 

Note that the dual of a $k$-angulation is a tree with $m$ vertices, one for each $k$-gon. Add $n-1$ labeled leaves to this tree, one for each edge of the $n$-gon other than the edge bisected by the axis of reflection, each connected to the vertex corresponding to the $k$-gon which uses that edge. If we consider the tree to be rooted at the vertex corresponding to the central $k$-gon, then the vertices of the tree are clearly divided into two identical halves, one for each side of the central $k$-gon. (See \figref{fig:coral overlap} for an example of such a construction.) A symmetric $k$-angulation is completely determined by one such half of the tree. Looking at just one half, we have a tree rooted at a vertex with degree $\frac{k-1}{2}$. Moreover, there are $\frac{m-1}{2}$ non-leaf vertices other than the root, each with degree $k$. Therefore it is a coral diagram of type $(k-1,\frac{k-1}{2},\frac{m-1}{2})$, as desired. All coral diagrams are obtained this way by exactly one symmetric $k$-angulation. 
\end{proof}

\begin{figure}
    \centering
    \begin{tikzpicture}[scale=2.5]
        \coordinate (1) at (-0.361242,0.932472);
        \coordinate (2) at (-0.673696,0.739009);
        \coordinate (3) at (-0.895163,0.445738);
        \coordinate (4) at (-0.995734,0.0922684);
        \coordinate (5) at (-0.961826,-0.273663);
        \coordinate (6) at (-0.798017,-0.602635);
        \coordinate (7) at (-0.526432,-0.850217);
        \coordinate (8) at (-0.18375,-0.982973);
        \coordinate (9) at (0.18375,-0.982973);
        \coordinate (10) at (0.526432,-0.850217);
        \coordinate (11) at (0.798017,-0.602635);
        \coordinate (12) at (0.961826,-0.273663);
        \coordinate (13) at (0.995734,0.0922684);
        \coordinate (14) at (0.895163,0.445738);
        \coordinate (15) at (0.673696,0.739009);
        \coordinate (16) at (0.361242,0.932472);
        \coordinate (0) at (0.,1.);
        
        \draw[thick,fill=orange,fill opacity=0.3] (0) -- (1) -- (2) -- (3) -- (4) -- (5) -- (6) -- (7) -- (8) -- (9) -- (10) -- (11) -- (12) -- (13) -- (14) -- (15) -- (16) -- (0);
        \draw[thick] (0) -- (7) -- (3);
        \draw[thick] (0) -- (10) -- (14);
        \draw[dashed] (0,1) -- (0,-0.982973);
        
	    \SetFancyGraph
    	\Vertex[NoLabel,x=-0.180621,y=0.966236]{0}
    	\Vertex[NoLabel,x=-0.517469,y=0.835741]{1}
    	\Vertex[NoLabel,x=-0.784429,y=0.592374]{2}
    	\Vertex[NoLabel,x=-0.945449,y=0.269003]{3}
    	\Vertex[NoLabel,x=-0.97878,y=-0.0906973]{4}
    	\Vertex[NoLabel,x=-0.879921,y=-0.438149]{5}
    	\Vertex[NoLabel,x=-0.662225,y=-0.726426]{6}
    	\Vertex[NoLabel,x=-0.355091,y=-0.916595]{7}
    	\Vertex[NoLabel,x=0.355091,y=-0.916595]{9}
    	\Vertex[NoLabel,x=0.662225,y=-0.726426]{10}
    	\Vertex[NoLabel,x=0.879921,y=-0.438149]{11}
    	\Vertex[NoLabel,x=0.97878,y=-0.0906973]{12}
    	\Vertex[NoLabel,x=0.945449,y=0.269003]{13}
    	\Vertex[NoLabel,x=0.784429,y=0.592374]{14}
    	\Vertex[NoLabel,x=0.517469,y=0.835741]{15}
    	\Vertex[NoLabel,x=0.180621,y=0.966236]{16}
    	
    	\Vertex[NoLabel,x=0,y=0]{origin}
    	\Vertex[NoLabel,x=-0.473865,y=0.198507]{L1}
    	\Vertex[NoLabel,x=-0.794474,y=-0.226047]{L2}
    	\Vertex[NoLabel,x=0.473865,y=0.198507]{R1}
    	\Vertex[NoLabel,x=0.794474,y=-0.226047]{R2}
    	
    	\Edges[style={gray}](origin,L1)
    	\Edges[style={gray}](origin,R1)
    	\Edges[style={gray}](origin,7)
    	\Edges[style={gray}](origin,9)
    	\Edges[style={gray}](L1,L2)
    	\Edges[style={gray}](L1,0)
    	\Edges[style={gray}](L1,1)
    	\Edges[style={gray}](L1,2)
    	\Edges[style={gray}](L2,3)
    	\Edges[style={gray}](L2,4)
    	\Edges[style={gray}](L2,5)
    	\Edges[style={gray}](L2,6)
    	\Edges[style={gray}](R1,R2)
    	\Edges[style={gray}](R1,14)
    	\Edges[style={gray}](R1,15)
    	\Edges[style={gray}](R1,16)
    	\Edges[style={gray}](R2,10)
    	\Edges[style={gray}](R2,11)
    	\Edges[style={gray}](R2,12)
    	\Edges[style={gray}](R2,13)
	\end{tikzpicture}
    \caption{The coral diagram for a particular symmetric pentagulation of the 17-gon as described in the proof of \thmref{thm:reflectioncounting}.}
    \label{fig:coral overlap}
\end{figure}

Now we proceed to count the number of $k$-angulations of an $n$-gon fixed by a rotation of order $d$. 
Because $k-2$ will be such an important quantity, we will work in terms of $s=k-2$ for the remainder of this section. Recall that $m=\frac{n-2}{k-2}$. 

\begin{theorem}\label{thm:rotationcounting}
Set $s$ and $m$ odd positive integers, both at least 3. Let $T(d,s,m)$ denote the number of $(s+2)$-angulations of an $(sm+2)$-gon which are fixed under a rotation of order $d$, where $d\mid(s+2)$. 
Then 
\[T(d,s,m) = \frac{sm+2}{s+2} R_{s+1,\frac{s+2}{d}}\lpr{\frac{m-1}{d}} = \binom{\frac{m(s+1) + 1}{d} - 1}{\frac{m-1}{d}}.\]
\end{theorem}
\begin{proof}

In any $(s+2)$-angulation, some $(s+2)$-gon contains the center of the $(sm+2)$-gon. Label the vertices of the polygon $\{1,\ldots, sm+2\}$ in clockwise order. In exactly $\frac{s+2}{sm+2}$ of rotationally symmetric $(s+2)$-angulations, the vertex labeled $sm+2$ is part of the central $(s+2)$-gon. Therefore by~\propref{prop:coralRaney}, it suffices to put the rotationally symmetric $(s+2)$-angulations where the central $(s+2)$-gon uses the vertex $sm+2$ in correspondence with coral diagrams of type $(s+1,\frac{s+2}{d},\frac{m-1}{d})$. Since the central $(s+2)$-gon has $d$-fold rotational symmetry, it also uses the vertex labeled $\frac{sm+2}{d}$, and has $\frac{s+2}{d}$ edges connecting $sm+2$ and $\frac{sm+2}{d}$. The entire $(s+2)$-angulation is determined by the $(s+2)$-gons incident to the vertices between $sm+2$ and $\frac{sm+2}{d}$.

Similar to the previous proof, we consider the dual of the $(s+2)$-angulation. Looking at just the tree rooted at the central $(s+2)$-gon with the $\frac{m-1}{d}$ other vertices corresponding to $(s+2)$-gons under the central $(s+2)$-gon in the interval betwen the vertices labeled $sm+2$ and $\frac{sm+2}{d}$, we see that the root has degree $\frac{s+2}{d}$. As in the proof of \thmref{thm:reflectioncounting}, add leaves with labels corresponding to the $\frac{sm+2}{d}$ outer edges of the polygon in this interval, so that each vertex corresponding to a non-central $(s+2)$-gon has degree $s+2$. Then we have constructed a type $(s+1,\frac{s+2}{d},\frac{m-1}{d})$ coral diagram which completely determines the $k$-angulation. Moreover, any coral diagram can be constructed this way.
\end{proof}

We now move to a new result regarding the parities of areas in a class of $(a,b)$-Dyck paths. 

\begin{theorem}\label{thm:dyckdiff}
$D_s(m)$, defined as the difference between the number of even-area and odd-area $(ms+1,m)$-Dyck paths, is equal to $$R_{s+1,\frac{s+1}{2}}\lpr{\frac{m-1}{2}}=\frac{1}{m}\binom{\frac{(s+1)m}{2}}{\frac{m-1}{2}}.$$
\end{theorem}

We shall first prove a sequence of results about $D_s(\ell,m)$, defined as the analogous quantity for paths in $Y_s(\ell,m)$. We begin with a pair of recurrences, and massage these in \propref{prop:dslm vanish} and \propref{prop:dslm raney} to obtain a recurrence for the Raney numbers. 

\begin{prop}\label{prop:dslm recur}
For odd $\ell$ and odd $s>3$, $D_s(\ell,m)$ satisfies the recurrences
\begin{align}
D_s(1,m) &= \sum\limits_{y=0}^{m-2}(-1)^{y+1}D_s(1,y)D_s(2s-1,m-(y+2))\label{dslm base}\\ 
D_s(\ell,m) &= \sum\limits_{y=0}^m(-1)^{(m+1)y}D_s(\ell-2,y)D_s(1,m-y)\label{dslm inductive}
\end{align}
along with the base case $D_s(\ell,0)=1$. 
\end{prop}
\begin{proof}[Proof of \eqref{dslm base}]
Our approach is inspired by the proof of \cite[Proposition 1.5]{reinerTennerYong2018}, where markers are placed on the Young diagram which form the basis for a recursion. 

In $Y_s(\ell,m)$, place a marker $y_i$ for $2\le i\le m-1$ in the box $(m-i,i-1,1)$ (so, with $2s-1$ empty boxes to its right) and place a marker $y_1$ in the box $(m-1,0,1)$. See \figref{fig:base yd}. 

\begin{figure}
    \centering
    \includegraphics[width=\textwidth]{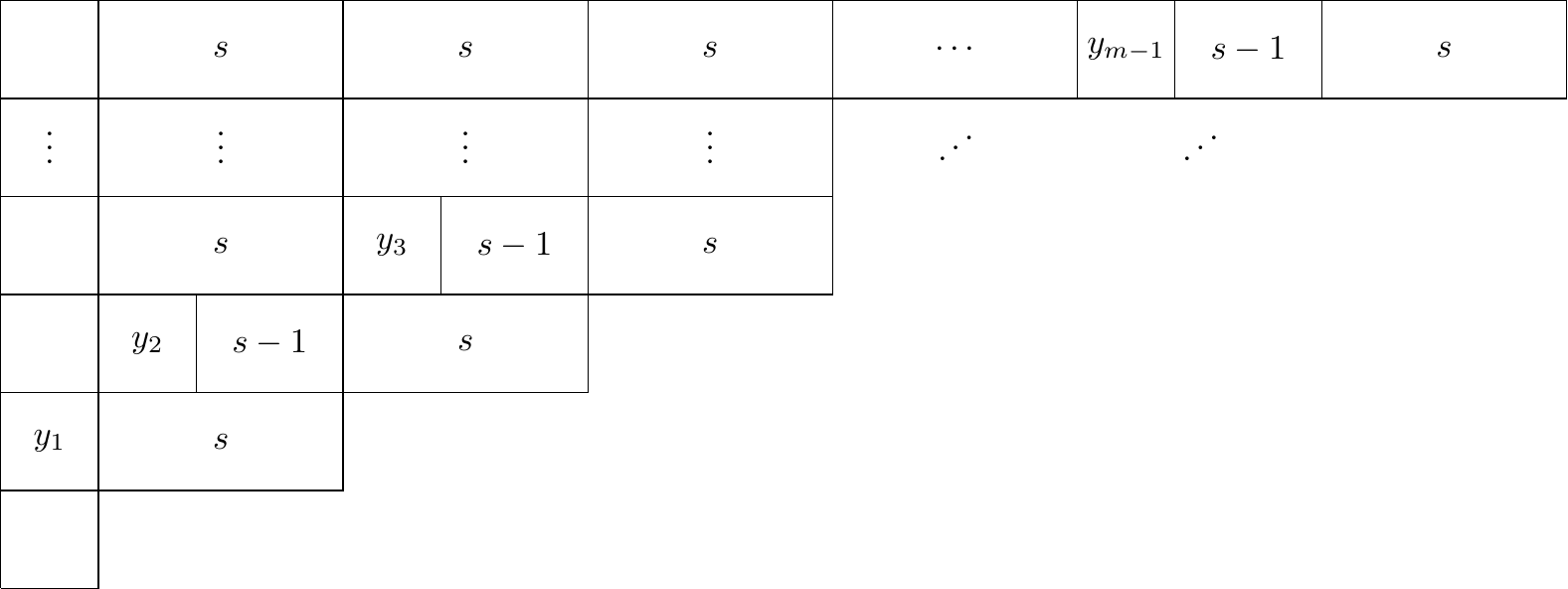}
    \caption{Young diagram $Y_s(1,m)$. }
    \label{fig:base yd}
\end{figure}

We sum over the southwesternmost marker contained above a given Dyck path in the diagram: $$\sum\limits_{y=2}^{m-1}(-1)^{e_s(1,m,y)}D_s(1,y-2)D_s(2s-1,m-y).$$ 
Here, the term $D_s(1,y-2)D_s(2s-1,m-y)$ obtains its first multiplicand from the possibilities for the path below the rectangle to the marker and its second from the possibilities for the path from the marker to the northeast corner. The sign is determined by the size of the $(m-y)\times(2+(y-1)s)$ rectangle included above the path. The $y=1$ term vanishes because any path containing the marker $y_1$ may be obtained by any path in $Y_s(1,m-2)$ prefixed by either $EN$ or $NE$, yielding paths whose areas cancel. 

A path in $Y_s(1,m)$ need not contain a marker; these are counted by noting that the first two directions are $NN$ followed by any path (of the same parity area) in $Y_s(1,m-2)$. Noting that $D_s(2s-1,0)=1$ and that $e_s(1,m,m)\equiv0\pmod{2}$, we find 
\begin{align}
D_s(1,m) &= \sum\limits_{y=2}^{m-1}(-1)^{e_s(1,m,y)}D_s(1,y-2)D_s(2s-1,m-y) + D_s(1,m-2)\nonumber\\
&= \sum\limits_{y=2}^{m-1}(-1)^{(m-y)\cdot(1+(y-1)s)}D_s(1,y-2)D_s(2s-1,m-y) + D_s(1,m-2)D_s(2s-1,0)\nonumber\\
&= \sum\limits_{y=2}^m(-1)^{y+1}D_s(1,y-2)D_s(2s-1,m-y)\nonumber\\
&= \sum\limits_{y=0}^{m-2}(-1)^{(m+y)(y+1)}D_s(1,y)D_s(2s-1,m-(y+2)).\nonumber
\end{align}
\end{proof}

\begin{proof}[Proof of \eqref{dslm inductive}]

Place a marker $y_i$ for $1\le i\le m-1$ in the box $(m-i,i,s-1)$ and place a marker $y_0$ in the box $(m,0,\ell-1)$ (so, in all cases with 1 empty box to its right). See \figref{fig:ind yd}.

\begin{figure}
    \centering
    \includegraphics[width=\textwidth]{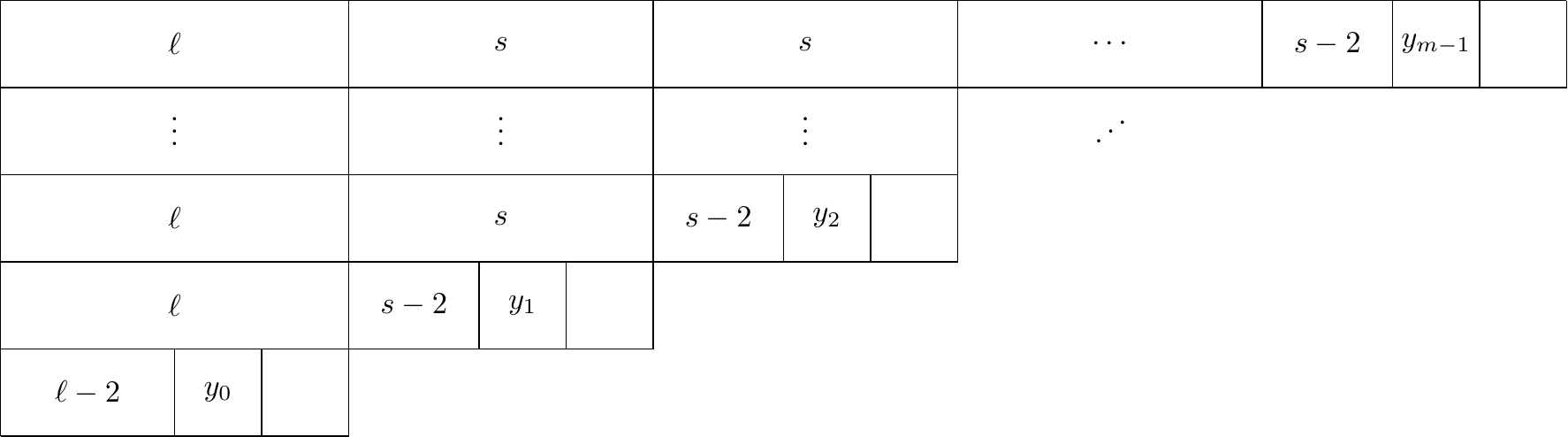}
    \caption{Young diagram $Y_s(\ell,m)$. }
    \label{fig:ind yd}
\end{figure}

We sum over the southwesternmost marker contained above a given Dyck path in the diagram: 
$$\sum\limits_{y=0}^{m-1}(-1)^{e_s(\ell,m,y)}D_s(\ell-2,y)D_s(1,m-y),$$ 
where each summand arises from the necessary inclusion of the $(m+1-y)\times(\ell+sy-1)$ rectangle (yielding the exponent $e_s$ of $-1$) along with a path below the rectangle to the marker (yielding the first multiplicand) and a path from the marker (yielding the second multiplicand). 

A path in $Y_s(\ell,m)$ need not contain a marker; these are counted by noting that the first direction is $N$ followed by any path (of opposite parity area) in $Y_s(s,m-1)$ (since each marker $y_i$ is avoided). Noting that $D_s(1,0)=1$ and that $e_s(1,m,m)\equiv0\pmod{2}$, we find 
\begin{align*}
D_s(\ell,m) &= \sum\limits_{y=0}^{m-1}(-1)^{e_s(\ell,m,y)}D_s(\ell-2,y)D_s(1,m-y) + D_s(\ell-2,m) \\
&= \sum\limits_{y=0}^{m-1}(-1)^{(m+1-y)\cdot(\ell+sy-1)}D_s(\ell-2,y)D_s(1,m-y) + D_s(\ell-2,m)D_s(1,0) \\
&= \sum\limits_{y=0}^m(-1)^{my}D_s(\ell-2,y)D_s(1,m-y).
\end{align*}
\end{proof}

\begin{prop}\label{prop:dslm vanish}
For all odd $\ell$, and all $s$ and $m$ with $s>3$, $D_s(\ell,m)$ vanishes. 
\end{prop}
\begin{proof}
We fix $s$ and induct on pairs of positive odd integers $(\ell,m)$ in the poset where $(\ell,m)\le(\ell',m')$ if and only if $\ell\le\ell'$ and $m\le m'$. We prove the result by direct computation for elements in the set $I \coloneqq \{(\ell,1)\mid\text{odd $\ell\in\N$}\}$, in which case $Y_s(\ell,1)$ is simply $\ell$ blocks in a (horizontal) row. There are $\ell+1$ paths, characterized by the position of the $N$, and as that position increases their areas alternate parity. $I$'s upwards-closed ideal equals the entire poset, so we invoke strong induction and use \propref{prop:dslm recur} \eqref{dslm base} and \eqref{dslm inductive}: we observe that for odd $m$ and any $y\in\Z$:  \begin{itemize}
    \item exactly one of $y$ and $m-y$ is odd; and
    \item exactly one of $y$ and $m-y-2$ is odd. 
\end{itemize}
The result follows. 
\end{proof}

\begin{prop}\label{prop:dslm raney}
We conclude that for even $m$, 
$$D_s(\ell,m)=R_{s+1,\frac{\ell+1}{2}}\lpr{\frac{m}{2}}.$$ 
\end{prop}
\begin{proof}
Since in \propref{prop:dslm recur} \eqref{dslm base} and \eqref{dslm inductive} the first argument of $D_s$ is always odd, for odd $y$ the summands vanish. Hence we restrict our attention to even $y$. Letting $D_s'(\ell,m) \coloneqq D_s(\ell,2m)$, we observe that:
\begin{align*}
D_s'\lpr{1,\frac{m}{2}} &= D_s(1,m) \\
&= \sum\limits_{y=0}^{\frac{m}{2}-1}(-1)^{2y+1}D_s(1,2y)D_s(2s-1,m-(2y+2)) \\
&=-\sum\limits_{y=0}^{\frac{m}{2}-1}D_s'(1,y)D_s'\lpr{2s-1,\frac{m}{2}-y-1}. \\
D_s'\lpr{\ell,\frac{m}{2}} &= D_s(\ell,m) \\
&= \sum\limits_{y=0}^{\frac{m}{2}}(-1)^{(m+1)2y}D_s(\ell-2,2y)D_s(1,m-2y) \\
&= \sum\limits_{y=0}^{\frac{m}{2}}D_s'(\ell-2,y)D_s'\lpr{1,\frac{m}{2}-y}.
\end{align*}
These precisely match the recurrences of \propref{prop:RaneyRecursion2}. 
\end{proof}

\begin{proof}[Proof of {\thmref{thm:dyckdiff}}]
The case of $s=1$ is proved in \cite[Theorem 6.1]{rs2017dihedral}, as the Raney numbers $R_{2,1}(n)$ are precisely the Catalan numbers. 

We observe that $D_s(m)=D_s(0,m)=D_s(s,m-1)$. Since $m$ is odd, from \propref{prop:dslm raney}, we are done. 
\end{proof}

From this result, we are able to claim a proof of the reflection case of dihedral sieving for the instance of $k$-angulations of an $n$-gon, since we have shown equality between the number of $k$-angulations fixed by a given reflection and the appropriate evaluation of the Catalan $q,t$-analogue. The following result of Eu and Fu will be used for the rotational case of dihedral sieving.

\begin{prop}[\cite{eufu2008dissections}]\label{prop:catalanatroots}
For $d\geq 2$ a divisor of $sm+2$, let $\omega$ be a primitive $d$th root of unity. Then
\[\Cat_{sm+1,m}(\omega, \omega^{-1}) = 
\begin{cases}
{\frac{m(s+1)+1}{d}-1\choose \frac{m-1}{d}} &\text{if }d\geq 2\text{ and }d|k,\\
0 & \text{otherwise.}
\end{cases}
\]
\end{prop}

Combining all of these results, we obtain: 

\begin{proof}[Proof of {\thmref{thm:Akangulations}}]
For reflections on the polygon, by \thmref{thm:reflectioncounting} and \thmref{thm:dyckdiff}, we have
\[S_{s+2}(n) = D_s(m) = \Cat_{sm+1, m}(1, -1)\]
where $S_{s+2}(n)$ is the number of $(s+2)$-dissections of an $n$-gon and $D_s(m)$ is the difference between the number of even-area and odd-area $(ms+1,m)$-Dyck paths. 

For rotations on the polygon, by \thmref{thm:rotationcounting} and \propref{prop:catalanatroots}, we have for $\omega$ a primitive $d$th root of unity
\[T(d,s,m) = \Cat_{sm+1, m}(\omega, \omega^{-1})\]
where $T(d,s,m)$ denotes the number of $(s+2)$-angulations of an $(sm+2)$-gon fixed under rotation by $\omega$. Now we can conclude dihedral sieving for $(X_{s,m}\ \acton \ I_2(n),\Cat_{sm+1,m}(q,t))$.
\end{proof}

\section{Dihedral Sieving on Clusters of Other Types}\label{sec:othertypes}

In this section we prove our second main result, which we recall here:

\setcounter{section}{1}
\setcounter{theorem}{2}
\begin{theorem}
Assume two conjectures of C. Stump \cite[Conjectures 3 and 4]{stump2010qtfcat}. Then, we have that the pair $(\Delta(\Phi) \ \acton \ I_2(n),\Cat(\Phi,q,t))$ exhibits dihedral sieving for all odd $n$ and $\Phi$ of type $A$, $B/C$, $D$, $E$, $F$, or $I$.
\end{theorem}
\setcounter{section}{4}
\setcounter{theorem}{0}

We make no such claims about cluster complexes for $H_3$ and $H_4$, since the former is of the wrong parity for odd dihedral sieving, and the latter lacks a root poset.

First, we define the polynomial we will use. 
In \cite{eufu2008dissections}, the polynomial \[ \Cat(\Phi,q) \coloneqq \prod\limits_{i=1}^n \frac{[h+e_i +1]_q}{[e_i + 1]_q}\]
\noindent is used to exhibit cyclic sieving. Here $h$ is the Coxeter number and $e_1,\ldots,e_n$ are the exponents of $\Phi$, though these details will not pertain directly to our discussion. Observe that when $\Phi$ is of type $A_n$, we have $h=n+1$ and $e_i=i+1$ from which we have $\Cat(\Phi,q)$ collapse to the usual $q$-Catalan numbers.

\begin{theorem}[{\cite[Theorem 1.2]{eufu2008dissections}}]\label{thm:cycalltypes}
The pair $(\Delta(\Phi)\ \acton \ C_n,\Cat(\Phi,q) )$ exhibits the cyclic sieving phenomenon for all root systems $\Phi$.
\end{theorem} 
Eu and Fu actually consider cyclic sieving for faces of the cluster complex of any dimension, whereas we only consider the top-dimensional faces, corresponding to $k$-angulations. An interesting direction to extend our result further would be to consider dihedral sieving for $k$-divisible dissections, corresponding to faces of any dimension. 

We need to introduce a $t$ parameter in order to exhibit dihedral sieving. We'll use the following definition, which includes assuming a conjecture of C. Stump~\cite{stump2010qtfcat}.

\begin{definition}[cf. {\cite[Conjectures 3 and 4]{stump2010qtfcat}}]\label{def:qtcatW} Let $\Phi$ be a root system with root poset $P$ and Weyl group $W$, which acts on an $n$-dimensional complex space $V$. Then $W$ acts diagonally on $\mb C[V\oplus V^*] = \mb C[\mathbf{x},\mathbf{y}]$. Let $\mc A$ be the ideal of $\mb C[\mathbf{x},\mathbf{y}]$ generated by the polynomials $p$ such that $w(p) = \det(w)p$ for all $w\in W$. The $q,t$-$\Cat(\Phi)$ polynomial is defined as a Hilbert series graded in $q$ and $t$:
\[\Cat(\Phi,q,t) \coloneqq \mc H(\mc A/\langle \mathbf{x},\mathbf{y}\rangle \mc A; q,t). \]

This polynomial conjecturally satisfies the following two specializations: \[\Cat(\Phi,q,q^{-1}) = \Cat(\Phi,q)\] and \[\Cat(\Phi,q,1) = \sum_{I \in J(P)} q^{|I|}.\]

For our purposes, the only relevant property of the $q,t$-$\Cat(\Phi)$ polynomials are these specializations, which we will assume to hold.
\end{definition}

The first specialization takes care of the cyclic part for us, due to \thmref{thm:cycalltypes}. It only remains for us to show the following claim, for each type $\Phi$.
\begin{claim}\label{claim:checkthis}
For a root system $\Phi$ with root poset $P$, and $\Delta(\Phi)$ with odd-$n$ dihedral action, we have that $\sum\limits_{I \in J(P)} (-1)^{|I|}$ counts the number of facets in $\Delta(\Phi)$ fixed by $\tau_+$, or by $\tau_-$.
\end{claim}
For $\Phi$ of type $I$, despite being non-crystallographic, there is still a nice choice of poset $P$ which we can call the root poset as shown in~\cite{CP2015root}, 
so we will do that case too. The remainder of this section is dedicated to verifying \claimref{claim:checkthis} for each type in turn. For types $B$ and $D$, we will first describe a realization of facets in $\Delta(\Phi)$ as certain decorated triangulations. 

\subsection{Type $B_n$/$C_n$ Cluster Complex}

In \cite{fr2005generalized}, a combinatorial realization of $\Delta(\Phi)$ for type $B_n$ is given as follows (and it is the same for $C_n$). Let $P$ be a centrally symmetric regular polygon with $2n+2$ vertices. The vertices of the complex are of two types:
\begin{itemize}
    \item The diameters of $P$, i.e., the diagonals connecting antipodal vertices;
    \item The pairs $(D, D')$ of distinct diagonals of $P$ such that $D$ is related to $D'$ by a half-turn about the center of $P$.
\end{itemize}
Two vertices are called noncrossing if no diagonal representing one vertex crosses a diagonal representing the second vertex. The faces of the complex are the sets of pairwise noncrossing vertices. Therefore the maximal faces correspond to centrally symmetric triangulations of $P$.

Moreover, the action of the reflections $\tau_-,\tau_+$ on a type $B_n$ face is in fact reflection in the usual sense \cite[Proposition 3.15 (1)]{fz2003Ysystems}.

See \figref{fig:b4} for an example. 

\begin{figure}
    \centering
    \includegraphics[scale=2]{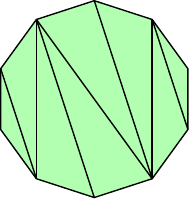}\quad\includegraphics[scale=2]{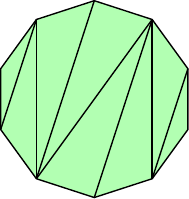}
    \caption{Two maximal faces for type $B_4/C_4$ acted on by $\tau_+$.}
    \label{fig:b4}
\end{figure}

\begin{lemma}\label{lem:noBfix}
If $F$ is a facet for $\Delta(B_n)$ for $n$ even and $\eps \in\{+,-\}$, then $\tau_\eps(F) \neq F$.  
\end{lemma}
\begin{proof}
We realize $F$ as a centrally symmetric triangulation of a $(2n+2)$-gon. Then $F$ must contain some diameter $D$, otherwise the central triangle would not be centrally symmetric by a half-turn.

Let $A$ be the axis of reflection of $\tau_\eps$. If $A$ does not go through vertices, then it must be perpendicular to the diameter contained in $F$, otherwise the reflection about $A$ cannot fix $F$. Now, any triangulation  of the $n+2$ vertices not on the right of $D$ must contain some diagonal which crosses the axis $A$ (because $n$ is even). This diagonal will not be sent to another diagonal under $\tau_\eps$, since if it was, then the configuration would not be non-crossing.

Next, suppose that $A$ does go through vertices. In order to fix the diameter, the axis $A$ and the diameter $D$ must coincide. Then the left and right sides of $D$ must be symmetric, in addition to having the half-turn central symmetry. This implies that the $F$ must also be fixed by the reflection perpendicular to $A$, which was covered in the previous case. Therefore $F$ cannot be fixed by $\tau_{\eps}$.
\end{proof}

Now it just remains to show that there are equally many order ideals $B_n^{+} = T_{n,2n}$ of each parity. (See \figref{trap poset ex} for an example of such a poset.)

\begin{figure}
\centering
\begin{tikzpicture}[scale=1.5]
	\SetFancyGraph
	\Vertex[NoLabel,x=0,y=0]{1}
	\Vertex[NoLabel,x=1,y=0]{2}
	\Vertex[NoLabel,x=2,y=0]{3}
	\Vertex[NoLabel,x=3,y=0]{4}
	\Vertex[NoLabel,x=0.5,y=0.5]{5}
	\Vertex[NoLabel,x=1.5,y=0.5]{6}
	\Vertex[NoLabel,x=2.5,y=0.5]{7}
	\Vertex[NoLabel,x=1,y=1]{8}
	\Vertex[NoLabel,x=2,y=1]{9}
	\Vertex[NoLabel,x=1.5,y=1.5]{10}
	\Vertex[NoLabel,x=0,y=1]{11}
	\Vertex[NoLabel,x=0.5,y=1.5]{12}
	\Vertex[NoLabel,x=0,y=2]{13}
	\Vertex[NoLabel,x=1,y=2]{14}
	\Vertex[NoLabel,x=0.5,y=2.5]{15}
	\Vertex[NoLabel,x=0,y=3]{16}
	\Edges[style={red,ultra thick}](1,5)
	\Edges[style={red,ultra thick}](2,5)
	\Edges[style={thick}](2,6)
	\Edges[style={thick}](3,6)
	\Edges[style={red,ultra thick}](3,7)
	\Edges[style={red,ultra thick}](4,7)
	\Edges[style={thick}](5,8)
	\Edges[style={thick}](6,8)
	\Edges[style={thick}](6,9)
	\Edges[style={thick}](7,9)
	\Edges[style={thick}](8,10)
	\Edges[style={thick}](9,10)
	\Edges[style={red,ultra thick}](5,11)
	\Edges[style={thick}](11,12)
	\Edges[style={thick}](8,12)
	\Edges[style={thick}](12,13)
	\Edges[style={thick}](12,14)
	\Edges[style={thick}](10,14)
	\Edges[style={thick}](13,15)
	\Edges[style={thick}](14,15)
	\Edges[style={thick}](15,16)
\end{tikzpicture}
\caption{$T_{4,8}$ with an order ideal highlighted in red. }
\label{trap poset ex}
\end{figure}
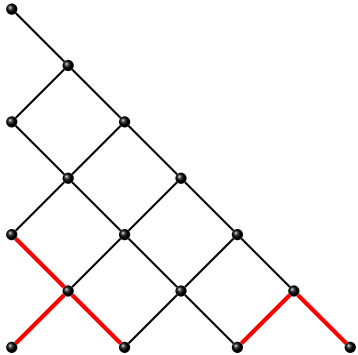

\begin{lemma}\label{lem:pB0}
For $T_{n,2n}$ a trapezoid poset and $J(T_{n,2n})$ its set of order ideals, $$\sum\limits_{I\in J(T_{n,2n})}(-1)^{\abs{I}}=0.$$
\end{lemma}
\begin{proof}
We proceed by induction on $n$. The case $n=1$ follows by computation. We note that the upper $\abs{T_{n-1,2n-2}}$-many elements of $T_{n,2n}$ are precisely the subposet $T\cong T_{n-1,2n-2}$, say via the poset injection $f$. For $I\in J(T_{n-1,2n-2})$, let $F(I)\subset T_{n,2n}\setminus T$ be defined such that for each $\ga\in F(I)$, there is $\gb\in I$ with $\ga\le f(\gb)$. (Elements of $F(I)$ are ``forced'' to appear in the order ideal of $T_{2n,n}$ corresponding to $f(I)$.) For $I\in J(T_{n-1,2n-2})$, there are two cases: $F(I)\subsetneq T_{n,2n}\setminus T$, and $F(I)=T_{n,2n}\setminus T$. 

In the former case, we can ``toggle'' any element $\ga\in T_{n,2n}\setminus(T\cup F(I))$, i.e., we consider all order ideals containing $f(I)\cup F(I)$ as well as additional elements of $T_{n,2n}\setminus T$ but not $\ga$, as well as all order ideals containing $f(I)\cup F(I)\cup\{\ga\}$ as well as additional elements of $T_{n,2n}\setminus T$. It is clear that these sets are of the same size and their elements pair up, with sizes of opposing parities. 

In the latter case, we may toggle any element of the bottom row of $T$ other that one with down-degree 1 to again obtain two sets of the same size whose elements pair up, having sizes of opposing parities. 
\end{proof}

\subsection{Type $D_n$ Cluster Complex}

In \cite{fr2005generalized}, a combinatorial realization of $\Delta(\Phi)$ for type $D_n$ is given as follows.
Let $P$ be a regular polygon with $2n$ vertices. The vertices of the combinatorial realization of $\Delta(D_n)$ will fall into two groups. The vertices in the first group corresponds one-to-one to pairs of distinct non-diameter diagonals in $P$ related by a half turn. In the second group, each vertex is indexed by a diameter of $P$, together with two \emph{flavors}, which we call ``red'' and ``blue.'' Thus each diameter occurs twice, in each of the two flavors. We label the vertices of $P$ counterclockwise by $\{1,\ldots, n, -1,\ldots, -n\}$ and we call $[1, -1]$ the {\em primary diameter}. By construction, the map $R$ acts by rotating $2$ clockwise to $1$ and switching the flavors of certain diagonals. Specifically, $R$ preserves flavor when applied to a diameter of form $[k,-k]$ for $1\leq k\leq n$ unless $k = 1$. If $k=1 $, the flavor is switched.

The notion of compatibility for the cluster complex of type $D_n$ is defined as follows. Two vertices at least one of which is not a diameter are compatible or not according to precisely the same rules as in the type $B_n$ case. A more complicated condition determines whether two diameters are compatible. Two diameters with the same location and different flavors are compatible. Two diameters at different locations are compatible if and only if applying $R$ repeatedly until either of them is in position $[1,-1]$ results in diameters of the same flavor. Explicitly, let $D$ be a flavored diameter and let $\overline{D}$ denote $D$ with its flavor reversed. Then for $1 \le k \le n-1$, the flavored diameter $R(D)$ is compatible with $D$ (and incompatible with $\overline{D}$) if and only if $k$ is even.

By definition, the maximal faces of the complex are the diagonals of $P$ that satisfy the following. Within the flavored diagonals, if they are not of same flavor, they are allowed to overlap; if they are of the same flavor, they are allowed to cross. Aside from within the flavored diagonals, flavored diagonals cannot cross non-flavored diagonals and non-flavored diagonals should be pairwise noncrossing.

In this case, the action of the reflections $\tau_-,\tau_+$ on a type $D_n$ face is not just to reflect the polygon, but also to reverse all of the flavors \cite[Proposition 3.16 (1)]{fz2003Ysystems}.

See \figref{fig:d5} for an example. 

\begin{figure}
    \centering
    \includegraphics[scale=2]{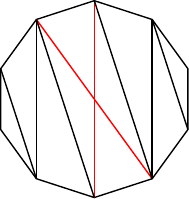}\quad\includegraphics[scale=2]{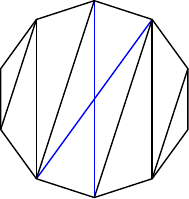}
    \caption{Two maximal faces in $D_5$ acted on by $\tau_+$.}
    \label{fig:d5}
\end{figure}

\begin{lemma}\label{lem:noDfix}
If $F$ is a facet for $\Delta(D_n)$ for $n$ odd and $\eps \in\{+,-\}$, then $\tau_\eps(F) \neq F$.  
\end{lemma}\begin{proof}
Suppose $F$ contains a diameter which exists in one flavor but not the other. Then the reflection $\tau_\eps$ causes there to exist a diameter in the other flavor with the same property. But only one flavor can have such a diameter in the original configuration, so $\tau_\eps$ does not fix $F$.

There must be at least one diameter, for the same reason as for type $B_n$. This diameter exists in both flavors. Then no other diameter may exist, because differently-flavored diameters cannot cross. Therefore we are in exactly the same setting as in the proof of \lemref{lem:noBfix}, so we can conclude.
\end{proof}

Again we are left to count the order ideals of the root poset by parity. The type-$D$ poset is as described in the following Lemma.

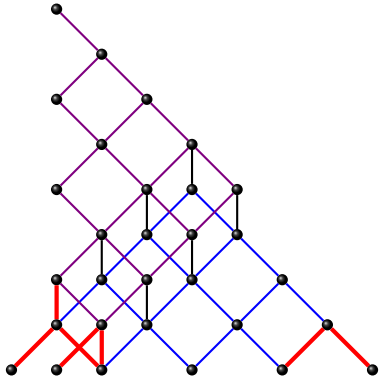
\begin{figure}
\centering
\begin{tikzpicture}[scale=1.2]
	\SetFancyGraph
	\Vertex[NoLabel,x=0,y=0]{1}
	\Vertex[NoLabel,x=1,y=0]{2}
	\Vertex[NoLabel,x=2,y=0]{3}
	\Vertex[NoLabel,x=3,y=0]{4}
	\Vertex[NoLabel,x=4,y=0]{5}
	\Vertex[NoLabel,x=0.5,y=0.5]{6}
	\Vertex[NoLabel,x=1.5,y=0.5]{7}
	\Vertex[NoLabel,x=2.5,y=0.5]{8}
	\Vertex[NoLabel,x=3.5,y=0.5]{9}
	\Vertex[NoLabel,x=1,y=1]{10}
	\Vertex[NoLabel,x=2,y=1]{11}
	\Vertex[NoLabel,x=3,y=1]{12}
	\Vertex[NoLabel,x=1.5,y=1.5]{13}
	\Vertex[NoLabel,x=2.5,y=1.5]{14}
	\Vertex[NoLabel,x=2,y=2]{15}
	
	\Edges[style={red,ultra thick}](1,6)
	\Edges[style={red,ultra thick}](2,6)
	\Edges[style={blue,thick}](2,7)
	\Edges[style={blue,thick}](3,7)
	\Edges[style={blue,thick}](3,8)
	\Edges[style={blue,thick}](4,8)
	\Edges[style={red,ultra thick}](4,9)
	\Edges[style={red,ultra thick}](5,9)
	\Edges[style={blue,thick}](6,10)
	\Edges[style={blue,thick}](7,10)
	\Edges[style={blue,thick}](7,11)
	\Edges[style={blue,thick}](8,11)
	\Edges[style={blue,thick}](8,12)
	\Edges[style={blue,thick}](9,12)
	\Edges[style={blue,thick}](10,13)
	\Edges[style={blue,thick}](11,13)
	\Edges[style={blue,thick}](11,14)
	\Edges[style={blue,thick}](12,14)
	\Edges[style={blue,thick}](13,15)
	\Edges[style={blue,thick}](14,15)

	\Vertex[NoLabel,x=0.5,y=0]{1'}
	\Vertex[NoLabel,x=1,y=0.5]{2'}
	\Vertex[NoLabel,x=1.5,y=1]{3'}
	\Vertex[NoLabel,x=2,y=1.5]{4'}
	\Vertex[NoLabel,x=2.5,y=2]{5'}
	\Vertex[NoLabel,x=0.5,y=1]{6'}
	\Vertex[NoLabel,x=1,y=1.5]{7'}
	\Vertex[NoLabel,x=1.5,y=2]{8'}
	\Vertex[NoLabel,x=2,y=2.5]{9'}
	\Vertex[NoLabel,x=0.5,y=2]{10'}
	\Vertex[NoLabel,x=1,y=2.5]{11'}
	\Vertex[NoLabel,x=1.5,y=3]{12'}
	\Vertex[NoLabel,x=0.5,y=3]{13'}
	\Vertex[NoLabel,x=1,y=3.5]{14'}
	\Vertex[NoLabel,x=0.5,y=4]{15'}
	
	\Edges[style={red,ultra thick}](1',2')
	\Edges[style={violet,thick}](2',3')
	\Edges[style={violet,thick}](3',4')
	\Edges[style={violet,thick}](4',5')
	\Edges[style={violet,thick}](6',7')
	\Edges[style={violet,thick}](7',8')
	\Edges[style={violet,thick}](8',9')
	\Edges[style={violet,thick}](10',11')
	\Edges[style={violet,thick}](11',12')
	\Edges[style={violet,thick}](13',14')
	\Edges[style={violet,thick}](15',14')
	\Edges[style={violet,thick}](14',12')
	\Edges[style={violet,thick}](12',9')
	\Edges[style={violet,thick}](9',5')
	\Edges[style={violet,thick}](13',11')
	\Edges[style={violet,thick}](11',8')
	\Edges[style={violet,thick}](8',4')
	\Edges[style={violet,thick}](10',7')
	\Edges[style={violet,thick}](7',3')
	\Edges[style={violet,thick}](6',2')
	
	\Edges[style={red,ultra thick}](6,6')
	\Edges[style={red,ultra thick}](2,2')
	\Edges[style={thick}](7,3')
	\Edges[style={thick}](10,7')
	\Edges[style={thick}](11,4')
	\Edges[style={thick}](13,8')
	\Edges[style={thick}](14,5')
	\Edges[style={thick}](15,9')
\end{tikzpicture}
\caption{An example of a double triangle poset, with the two triangle posets colored differently and an order ideal indicated with thicker edges.}
\label{double triangle poste}
\end{figure}

\begin{lemma}
For $\hat{T}_n$ a {\em double triangle poset} --- i.e., two mirror-image copies of the same triangle poset, one placed over the other such that exactly two diagonal rows of nodes lie ``over'' each other and are correspondingly connected (see \figref{double triangle poste}) --- and $J(\hat{T}_n)$ its set of order ideals, we have that $$\sum\limits_{I\in J\lpr{\hat{T}_n}}(-1)^{\abs{I}}=0.$$
\end{lemma}
\begin{proof}
There is a copy $T$ of $\hat{T}_{n-1}$ lying inside of $\hat{T}_n$, which is $\hat{T}_n$ without its minimal and rightmost elements. 
Let $F(I)$ be defined as in the proof of \lemref{lem:pB0}, and let $O(I)\subset\hat{T}_n\setminus\hat{T}_{n-1}$ be defined such that for each $\ga\in O(I)$, there is $\ga\in I$ with $f(\gb)\le\ga$. Elements of $O(I)$ are ``optional'' to include in an order ideal of $\hat{T}_n$ containing $I$. Similarly to the proof of \lemref{lem:pB0}, we pair up order ideals which differ in size by one by toggling a single element. 

For $I\in J\lpr{\hat{T}_{n-1}}$, there are two cases: $F(I)\subsetneq\hat{T}_n\setminus T$ and $F(I)=\hat{T}_n\setminus T$. In the former case, we can toggle any minimal element $\ga\in\hat{T}_n\setminus(T\cup F(I))$. In the latter case, we may toggle any element of the bottom row of $T$ other than the one with down-degree 1. Therefore by pairing up order ideals of opposing parities, the sum vanishes. 
\end{proof}

\subsection{Type $I$} 

The type $I$ root system is not crystallographic, however, the construction of a cluster complex described in~\secref{sec:ccdefs} can still be carried out, as noted in \cite{fomin2005root} and \cite[\S2]{fr2005generalized}. We view the action on Type $I$'s cluster complex simply as the action of $\<\tau_-,\tau_+\>$ on $\Phi_{\ge-1}$, without any further combinatorial equivalences. The clusters are radially consecutive pairs of roots.

Arbitrarily label the simple roots $\ga_-$ and $\ga_+$. We shall show, without loss of generality, that $\tau_-$ has exactly one fixed cluster. Note that $-\ga_+$ is the only element of $\Phi_{\ge-1}$ on the same half-plane delineated by $\Span\{\ga_-\}$. Now, $I_-=\{\ga_-\}$, so $\tau_-=\gs_{\ga_-}$ and $-\ga_+$ is fixed while $\ga_-$ and $-\ga_-$ exchange places, in turn exchanging their clusters. Similarly, the other cluster containing $\ga_+$ switches places with the other cluster containing $-\ga_+$. Indeed, the only instance not nontrivially affected is the cluster whose positive span contains the normal vector to $\Span\{\ga_-\}$ not lying in the same half-plane as $-\ga_+$. Thus precisely this one cluster is fixed. 

It remains to count the order ideals by parity for type $I$. Here the poset comes from the construction in e.g. \cite{CP2015root}. See \figref{i2 poset ex} for an example. 

\begin{figure}
\centering
\begin{tikzpicture}[scale=1.5]
	\SetFancyGraph
	\Vertex[NoLabel,x=0,y=0]{1a}
	\Vertex[NoLabel,x=1,y=0]{1b}
	\Vertex[NoLabel,x=0.5,y=0.5]{2}
    \Vertex[NoLabel,x=0.5,y=1.5]{3}
	\Vertex[NoLabel,x=0.5,y=2.5]{4}
	\Vertex[NoLabel,x=0.5,y=3.5]{5}
	\Vertex[NoLabel,x=0.5,y=3.5]{6}
	\Edges[style={red,ultra thick}](1a,2)
	\Edges[style={red,ultra thick}](1b,2)
	\Edges[style={red,ultra thick}](2,3)
	\Edges[style={thick}](3,4)
	\Edges[style={thick}](4,5)
	\Edges[style={thick}](5,6)
\end{tikzpicture}
\caption{$I^{(7)}$ with an order ideal highlighted in red.}
\label{i2 poset ex}
\end{figure}
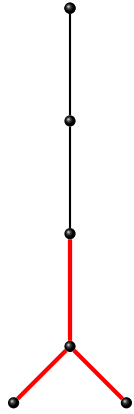

\begin{lemma}
Let $I^{(n)}$ be the line poset with two incomparable minimal elements, having $n$ total elements, with $J(I^{(n)})$ its set of order ideals. For odd $n$, $$\sum_{I\in J(I^{(n)})}(-1)^{\abs{I}}=-1.$$
\end{lemma}
\begin{proof}
There is one order ideal of size 0 (the empty one), two of size 1 (one for each minimal element), one of size 3 (the minimal ones along with their parent), one of size 4, etc. Summing, we have $$1-2+\underbrace{1-1+\cdots +1 -1}_{\text{$n-1$ terms}}=-1.$$
\end{proof}

\subsection{Exceptional Types}

For the exceptional types $E_6,E_7,E_8,F_4$, we were able to compute both the polynomial and the number of facets fixed by reflection using Sage. This data is found in \figref{fig:exceptional data}.

\begin{figure}
    \centering
    \begin{tabular}{|r|c|c|c|c|}
    \hline $W$ & $E_6$ & $E_7$ & $E_8$ & $F_4$ \\\hline\hline
    $\Cat(W,1,-1)$ & $-5$ & 0 & 14 & 1 \\\hline
    \# fixed by $\tau_\eps$ & 5 & $0,24$ & 14 & 1 \\\hline
    \end{tabular}
    \caption{Data for symmetric case of dihedral sieving for exceptional types. }
    \label{fig:exceptional data}
\end{figure}

Since we are only able to show odd dihedral sieving, the difference between $\tau_+$ and $\tau_i$ for type $E_7$ is no surprise. We have now verified (up to sign) \claimref{claim:checkthis} for all the desired types, which completes the proof of \thmref{thm:alltypesoftriangles}.

\section{Even More Sieving}\label{sec:evenmore}

\subsection{Even-$n$ Dihedral Sieving}

As Rao and Suk remarked \cite[\S7]{rs2017dihedral}, the dihedral group $I_2(n) = \< r,s | r^n = s^2 = 1,\ rs = sr^{-1} \>$ is quite different for $n$ even. Most notably, there are two distinct, non-conjugate reflections $s$ and $rs$. Viewed as planar reflections on an $n$-gon, say that $s$ is a reflection whose reflecting line goes through two vertices, and $rs$ is a reflection whose reflecting line bisects two sides. The first instance of even-$n$ dihedral sieving ought to be the simplest $I_2(n)$-action, namely, acting on the set $X = [n]$. In the case of the cyclic group and the dihedral group $I_2(n)$, this example gave us the $q$-analogue and the $q,t$-analogue for $n$, respectively. For even $n$, we need to distinguish between the group elements $s$ and $rs$ somehow. The one-dimensional representation which does this is called $\chi_b$ in \cite{rs2017dihedral}, and is defined as follows: \[\chi_b(g) \coloneqq \begin{cases}\phantom{-}1 & g \in \< r^2,s \> \\ -1 & g \notin \< r^2,s \> \end{cases} .\]
When $n$ is odd, we have $ \< r^2,s \> = I_2(n)$, so this representation is nontrivial only in the even-$n$ case. The need to utilize this representation as well suggests that we need to use three-variable polynomials for even-$n$ dihedral sieving. The following definition does not quite fit into the framework of \defref{def:genSieve}, because it uses multiple representations, but it does still generalize our initial notion because in the odd case, we always have $b = 1$.

\begin{definition}[Dihedral Sieving Phenomenon]\label{def:evendisieve}
Suppose $X$ is a finite set acted on by the dihedral group $I_2(n)$, and $X(q,t,b)$ is a polynomial which is symmetric in $q$ and $t$. The pair $(X\ \acton \ I_2(n), X(q,t,b))$ has the \emph{dihedral sieving phenomenon} if for all $g \in I_2(n)$ with eigenvalues $\{\lambda_1,\lambda_2\}$ for $\rho_{\de}(g)$ and $\lambda_3$ for $\chi_b(g)$:

\[|\{x\in X : gx = x\}| = X(\lambda_1,\lambda_2,\lambda_3).\]
\end{definition}

With this definition in mind, we can define a sort of $q,t,b$-analogue of $n$ as follows:

\[\<n\>_{q,t,b}\coloneqq \begin{cases} \{n\}_{q,t} & \nu_2(n) = 0 \\ \lcr{\frac{n}{2}}_{q,t}+b\cdot \lcr{\frac{n}{2}}_{q,t} & \nu_2(n) = 1 \\ \lcr{\frac{n}{2}}_{q,t} + b \cdot \lcr{\frac{n}{4}}_{q,t} + b^2 \cdot \lcr{\frac{n}{4}}_{q,t} & \nu_2(n) = 2 \\ \lcr{\frac{n}{2}}_{q,t} + b \cdot \lcr{\frac{n}{4}+1}_{q,t} +  b^2 \cdot \lcr{\frac{n}{4}-1}_{q,t} & \nu_2(n) \geq 3, \end{cases}\]
\noindent
where $\nu_2(n)$ denotes the 2-adic valuation of $n$, the greatest exponent $\nu$ of 2 such that $2^\nu\mid n$. 
Since $I_2(n)$ is a subgroup of $S_n$, we can use instances of symmetric sieving to get instances of dihedral sieving. This $q,t,b$-analogue was defined with the aim of having the $n$ terms in the summation $\<n\>_{q,t,b}$ be exactly the eigenvalues of $\rho_{\perm}(g)$, when evaluated at $(q,t,b) = (\lambda_1,\lambda_2,\lambda_3)$ as in \defref{def:evendisieve}. 

\begin{definition}[Plethystic Substitution]
For a finite sum $S=\sum\limits_{i=1}^N s_i$ of monomials $s_i$ in a polynomial ring, such as $\Z[q,t,b]$, and a symmetric function $f(x_1,x_2,\dots)$, the \emph{plethystic substitution} of $S$ into $f$ is $f([S])=f(s_1,s_2,\dots,s_N,0,0,\ldots)$. 
\end{definition}

\begin{prop}\label{prop:evenmultisets}
The pair $\left(\multiset{[n]}{k}\ \acton \ I_2(n) , h_k([\<n\>_{q,t,b}])   \right) $ exhibits dihedral sieving for all $n \geq k \geq 0$. Here $h_k$ is the complete homogeneous polynomial in $n$ variables and $[\<n\>_{q,t,b}]$ denotes the plethystic substitution of $\<n\>_{q,t,b}$.
\end{prop}
\begin{proof}[Proof Sketch]
A straightforward exhaustive check shows that for all $g\in I_2(n)$ the list $[\<n\>_{\lambda_1,\lambda_2,\lambda_3}]$ is exactly the list of eigenvalues of $g$ considered as an $n \times n$ permutation matrix. The result then follows from \propref{prop:symmulti}.
\end{proof}

While the above result is encouraging, it does not seem that the polynomial $h_k([\<n\>_{q,t,b}])$ lends itself very well to product formulas such as the $q,t$ binomial coefficient $\qtbinom{n}{k}_{q,t}$. Further work is needed to understand what, if any, is the ``correct'' notion of dihedral sieving when $n$ is even. One major shortcoming of the current definition of $\< n \>_{q,t,b}$ is that the expression seems particularly arbitrary when $\nu_2(n) \geq 3$. We would rather have a definition which is either independent of $2$-valuation or is defined recursively on the $2$-valuation.

\subsection{Symmetric Sieving for $k$-subsets}

Part of the reason we are interested in dihedral sieving is because often when a set has a natural cyclic action, it also has a dihedral action. In a few cases, such as $X = \binom{[n]}{k}$ ($k$-subsets) and $X = \multiset{[n]}{k}$ (the set of $k$-multisubsets), the symmetric group $S_n$ has the most natural action. We can define symmetric group sieving as another special case of \defref{def:genSieve}, using the \emph{permutation representation} $\rho \coloneqq \rho_{\perm}$. This $n$-dimensional representation sends each $\pi \in S_n$ to its permutation matrix over $\mb C$.

\begin{definition}[Symmetric Sieving]\label{def:symsieve}
Suppose that $S_n$ acts on a finite set $X$, and $X(q_1,\ldots,q_n)$ is a symmetric polynomial in $n$ variables. Then the pair $(X \ \acton \  S_n,X(q_1,\ldots,q_n))$ has \emph{symmetric sieving} if, for all permutations $g \in S_n$ with eigenvalues $\{\lambda_1,\ldots,\lambda_n\}$ for $\rho_{\perm}(g)$, \[|\{x \in X : gx = x\}| = X(\lambda_1,\ldots,\lambda_n). \]
\end{definition}

The symmetric group $S_n$ contains as a subgroup every group $G$ acting faithfully on $[n]$. For such groups $G$, instances of $G$-sieving on sets $X$ with a symmetric action can be obtained from instances of symmetric sieving. We present one instance of symmetric sieving, which will make a reappearance in \secref{sec:evenmore}.

\begin{prop}\label{prop:symmulti}
Let $ h_k(q_1,\ldots,q_n)$ be the degree-$k$ complete homogeneous polynomial in $n$ variables. Then the pair $\left(\multiset{[n]}{k} \ \acton \  S_n,h_k(q_1,\ldots,q_n)\right)$ exhibits symmetric sieving.
\end{prop}
\begin{proof}
Let $X = \multiset{[n]}{k}$ and $\rho$ be the $\GL_n(\mb C)$-representation on $\mb C[X]$ defined by acting on the variables $(x_1,\ldots,x_n)$ used in the basis for $\mb C[X]$ given by the degree $k$-monomials. Since $S_n$ is a subgroup of $\GL_n(\mb C)$ via $\rho_{\perm}$, the permutation representation, $\rho$ is also the $S_n$-representation on $\mb C[X]$. Fix $\pi \in S_n$, with eigenvalues $\underline{\lambda} = (\lambda_1,\ldots,\lambda_n)$ for $\rho_{\perm}(\pi)$. We need to show that $\tr \rho(\pi) = h_k(\lambda_1,\ldots,\lambda_n)$. Well, $\rho_{\perm} (\pi)$ may be diagonalized to the matrix diag$(\lambda_1,\ldots,\lambda_n) \eqqcolon M$, and so $\tr \rho(M) = \tr \rho(\rho_{\perm}(\pi))$, since $\rho_{\perm}(\pi)$ and $M$ are in the same conjugacy class in $\GL_n(\mb C)$. Finally, note that the eigenvalues of $\rho(M)$ are exactly the terms in $h_k(x_1,\ldots,x_n)$, so the result follows.
\end{proof}

In \secref{sect:phenom} we defined symmetric sieving and used the complete homogeneous polynomials for the case of $k$-multisubsets of $[n]$. A natural guess for symmetric sieving on $k$-subsets would be to use the elementary symmetric polynomials, but in fact these fail even in small cases. The correct polynomial is more complicated.

For an $n$-dimensional vector $\vec v = (x_1,\ldots,x_n)$ and non-negative integers $a$ and $b$, we define $s_{a,b}(x_1,\ldots,x_n)$ to be the \emph{Schur polynomial} $s_\lambda(x_1,\ldots,x_n)$, for the (hook) shape $\lambda = (a,1,1,\ldots,1)$ with $b$ total parts, unless $a$ or $b$ is zero, in which case we set $s_{a,0} = s_{0,b} = 0$. Then let $p_k(x_1,\ldots,x_n)$ be the symmetric polynomial defined by \[p_k(x_1,\ldots,x_n) =  \sum_{j = 0}^{\floor{k/2}} (-1)^{j}\left( s_{k-2j+1,j}(x_1,\ldots,x_n) + s_{k-2j,j+1}(x_1,\ldots,x_n)\right).\] 
We do not know of a more natural way to express this polynomial, though one may exist. 
\begin{theorem}\label{con:symsub}
The pair $\left(\binom{[n]}{k}\ \acton \ S_n,p_k(x_1,\ldots,x_n) \right)$ exhibits symmetric sieving for all $n \geq k \geq 0$.
\end{theorem}
We originally made this conjecture based on experimental evidence, and we have since been shown a proof of this fact by Christopher Ryba. Using the plethystic substitution as we did for \propref{prop:evenmultisets}, this result implies an instance of dihedral sieving for even $n$ on $\binom{[n]}{k}$. 

\section{Acknowledgements}

This research was performed at the University of Minnesota--Twin Cities REU with the support of NSF grant DMS-1148634. ZS is also supported by NSF Graduate Research Fellowship grant DGE-1752814. We are extremely grateful to Vic Reiner for his guidance and support throughout the course of this project, and would like to thank Sarah Brauner and Andy Hardt for their insightful feedback. We also thank Chris Ryba for sharing with us his proof of \thmref{con:symsub}, and Sam Hopkins for sharing with us an instance of dihedral sieving on plane partitions and for assistance with setting figures. Lastly, we are grateful to two anonymous reviewers for comments and feedback. 

\printbibliography

\end{document}